\documentclass[12pt,reqno]{amsart}

\DeclareMathOperator{\ann}{ann}%
\DeclareMathOperator{\rank}{rank}%

\usepackage{amssymb}
\usepackage{hyperref}
\usepackage{epigraph}
\usepackage{tikz}

\newcommand{\F}{\mathbb{F}}
\newcommand{\frt}{\frac{1}{2}}
\newcommand{\frp}{\frac{1}{p}}
\newcommand{\Fp}{\F_p}
\newcommand{\Ft}{\F_2}
\newcommand{\Gal}{\text{\rm Gal}}

\newcommand{\K}{\mathcal{K}}
\newcommand{\Lcc}{\mathcal{L}}
\newcommand{\N}{\mathbb{N}}

\newcommand{\Tc}{\mathcal{T}}
\newcommand{\Xcc}{\mathcal{X}}
\newcommand{\Ycc}{\mathcal{Y}}
\newcommand{\Z}{\mathbb{Z}}
\newcommand{\comment}[1]{}
\newcommand{\codim}{\text{\rm codim}}

\begin{document}

%\thanks{}

\title[Structure of units modulo $p^m$]{Galois module structure of the units modulo $p^m$ of cyclic extensions of degree $p^n$}

\author[J\'{a}n Min\'{a}\v{c}]{J\'{a}n Min\'{a}\v{c}}
\address{Department of Mathematics, Western University, London, Ontario, Canada N6A 5B7}
\email{minac@uwo.ca}
\thanks{The first author is partially supported by the Natural Sciences and Engineering Research Council of Canada grant R0370A01.  He also gratefully acknowledges the Faculty of Science Distinguished Research Professorship, Western Science, in years 2004/2005 and 2020/2021. The second author was partially supported by 2017--2019 Wellesley College Faculty Awards. The third author was supported  in part by National Security Agency grant MDA904-02-1-0061.}

\author[Andrew Schultz]{Andrew Schultz}
\address{Department of Mathematics, Wellesley College, 106 Central Street, Wellesley, MA \ 02481 \ USA}
\email{andrew.c.schultz@gmail.com}

\author[John Swallow]{John Swallow}
\address{Office of the President, Carthage College, 2001 Alford Park Drive, Kenosha, WI \ 53140 \ USA}
\email{jswallow@carthage.edu}

\begin{abstract}
Let $p$ be prime, and $n,m \in \mathbb{N}$.  When $K/F$ is a cyclic extension of degree $p^n$, we determine the $\mathbb{Z}/p^m\mathbb{Z}[\Gal(K/F)]$-module structure of $K^\times/K^{\times p^m}$. With at most one exception, each indecomposable summand is cyclic and free over some quotient group of $\Gal(K/F)$.  For fixed values of $m$ and $n$, there are only finitely many possible isomorphism classes for the non-free indecomposable summand.  

These Galois modules act as parameterizing spaces for solutions to certain inverse Galois problems, and therefore this module computation provides insight into the structure of absolute Galois groups.  More immediately, however, these results show that Galois cohomology is a context in which seemingly difficult module decompositions can practically be achieved: when $m,n>1$ the modular representation theory allows for an infinite number of indecomposable summands (with no known classification of indecomposable types), and yet the main result of this paper provides a complete decomposition over an infinite family of modules.
\end{abstract}

% There is a quote of Dirac ``A physical law must possess mathematical beauty."

\date{\today}

\maketitle

\newtheorem*{theorem*}{Theorem}
\newtheorem*{lemma*}{Lemma}
\newtheorem{theorem}{Theorem}
\newtheorem{proposition}{Proposition}[section]
\newtheorem{corollary}[proposition]{Corollary}
\newtheorem{lemma}[proposition]{Lemma}

\theoremstyle{definition}
\newtheorem*{definition*}{Definition}
\newtheorem*{remark*}{Remark}
\newtheorem*{example*}{Example}

%\parskip=10pt plus 2pt minus 2pt

%{\parskip=3pt \tableofcontents}

% Ireland, Rosen Classical introduction to modern number theory, 2nd edition, volume 84

% look up John's NSF grant (Proc London) and mention my Doob professorship and my Faculty Award from Wellesley

\epigraph{A physical law must possess mathematical beauty.}{Paul Dirac}

\section{Introduction and Main Theorems}\label{sec:main.theorems}
\subsection{Background and motivation}

One of the deepest problems in modern algebra is to understand absolute Galois groups. For a field $K$, the absolute Galois group $G_K$ is the profinite group associated to the extension $K_{\tiny{\text{sep}}}/K$, where $K_{\tiny{\text{sep}}}$ is the separable closure of $K$.  Though elegant calculations of $G_K$ exist for certain well-chosen fields $K$, the question is unresolved in general.  The case $K = \mathbb{Q}$, for example, is the subject of tremendous interest, and has therefore become a wellspring of considerable mathematical ingenuity.  Since the computation of absolute Galois groups is so challenging, one can pursue the more tractable goal of asking whether (and how) absolute Galois groups distinguish themselves within the collection of profinite groups.  For example, a classical result of Artin and Schreier (see \cite{AS1,AS2}) tells us that any nontrivial finite subgroup of an absolute Galois group must be isomorphic to $\mathbb{Z}/2\mathbb{Z}$.  (The reader who is interested in reading more about infinite Galois groups can consult \cite{Bourbaki}.)

For a given prime $p$, one can also study an analog of the absolute Galois group that serves as the classifying object for Galois $p$-groups over $K$.  If we write $K(p)$ for the maximal pro-$p$ extension of $K$ within $K_{\tiny{\text{sep}}}$, then we can define the absolute $p$-Galois group of $K$ as $G_K(p) := \Gal(K(p)/K)$.  It is a pro-$p$ group.  As with absolute Galois groups, one naturally wonders whether (and how) absolute $p$-Galois groups distinguish themselves within the collection of pro-$p$ groups.  There is much to say in this vein.   For example, Becker proved in \cite{Becker} an analog of the result of Artin and Schrier for absolute $p$-Galois groups.   See \cite{FriedJarden,Jarden.infiniteGalois,Koch,NSW} for fundamental results on Galois pro-$p$ groups. (Note that it is also very interesting to study the $p$-Sylow subgroups of $G_K$.  The interested reader can consult \cite{BarySorokerJardenNeftin}.)  Further progress on the study of pro-$p$ groups (via rigid fields, Galois modules, Demu\v{s}kin groups, etc.) has also been explored in \cite{KlopschSnopce,NosedaSnopceSerre,ShustermanZalesskii,SnopceZalesskii}.

Some fascinating, definite results were obtained in the case where $K$ is a finite extension of the $p$-adic numbers $\mathbb{Q}_p$.  Such a field is an example of a local field of characteristic $0$.  (See, for example, \cite{Gouvea, Guillot, Koblitz} for expositions on the basic properties of $p$-adic numbers.   For more advanced topics on local fields, see \cite{Serre-localfields}.) In 1947, Shafarevich proved the beautiful result that if $K$ is a local field of characteristic $0$ which does not contain a primitive $p$th root of unity, then $G_K(p)$ is a free pro-$p$ group (\cite{Shaf}).  In the case of local fields $K$ of characteristic $0$ which do contain a primitive $p$th root of unity, the situation is more complicated.  It has been clarified only over a number of years by several authors, including Kawada (\cite{Kawada}), Demu\v{s}kin (\cite{Dem1,Dem2,Dem3}), Serre (\cite{Serre.bourbaki}), and Labute (\cite{Labute.thesis}).  Serre showed that in this case $G_K(p)$ is a Demu\v{s}kin group. This means that $\dim_{\mathbb{F}_p}H^1(G_K(p),\mathbb{F}_p)$ is finite, $\dim_{\mathbb{F}_p}H^2(G_K(p),\mathbb{F}_p)=1$, and the cup product $H^1(G_K(p),\mathbb{F}_p) \times H^1(G_K(p),\mathbb{F}_p) \to H^2(G_K(p),\mathbb{F}_p)$ gives a non-degenerate skew-symmetric pairing. In his thesis, Labute was able to completely classify Demu\v{s}kin groups and write down their presentation in terms of generators and relations.  One of the major goals in current Galois theory is to extend this classification from Demu\v{s}kin group to the classification of all finitely generated absolute $p$-Galois groups.  For some details about this work, see \cite{Efrat, Marshall, MSpira, MT-Advances}.  

There are numerous methodologies in use to shed light on the ``special" structural properties of absolute $p$-Galois groups, including the study of Massey products (see \cite{EfratMatzri,GMT,HW,MRT,MT-JLMS,MT-JAlg,MT-JEMS}) and the Koszulity of Galois cohomology (see \cite{MPPT,MPQT,P}).  In this paper, however, we will pursue a particular invariant attached to absolute $p$-Galois groups that can be motivated easily using group-theoretic intuition.  

We know that $G_K(p)$ has some maximal elementary $p$-abelian quotient, and by Galois theory this corresponds to the maximal elementary $p$-abelian extension of $K$.  Assuming the appropriate roots of unity, Kummer theory tells us that this field corresponds to the $\mathbb{F}_p$-space $J(K):=K^\times/K^{\times p}$. When $K$ is itself a Galois extension of some field $F$ (with $\Gal(K/F)$ a $p$-group, just to stay focused on the category of $p$-groups), the natural action of $\Gal(K/F)$ on $K^\times$ descends to $J(K)$, making this $\mathbb{F}_p$-space into an $\mathbb{F}_p[\Gal(K/F)]$-module.  We can then ask: what is the structure of this module? how does it compare to the structure of a generic module over the group ring $\mathbb{F}_p[\Gal(K/F)]$?  

This question has been resolved in the case that $\Gal(K/F)$ is a cyclic $p$-group, both when $K$ contains roots of unity or not (see \cite{Bo,F} for early work over local fields, and \cite{MSS,MS} for the general case). Variants of this same idea have also been explored: an analog in the case of $\text{char}(K)=p$ has been considered in \cite{BS,Schultz}, where the Kummer module $K^\times/K^{\times p}$ is replaced by the Artin-Schreier module $K/\wp(K)$ (with $\wp(K) = \{k \in K: k=\ell^p-\ell \text{ for some }\ell \in K\}$); a cohomological generalization --- motivated by the fact that  $K^\times/K^{\times p} \simeq H^1(G_K(p),\mathbb{F}_p)$ in the presence of appropriate roots of unity --- is computed in \cite{LMSS,LMS} ; and in \cite{BLMS} the two previous ideas are hybridized, with quotients of Milnor $K$-groups used in place of cohomology groups when $\text{char}(K)=p$.  Recently in \cite{CMSS},  the module structure for $K^\times/K^{\times 2}$ was  computed when $\Gal(K/F) \simeq \mathbb{Z}/2\mathbb{Z} \oplus \mathbb{Z}/2\mathbb{Z}$.  A forthcoming paper (\cite{HMS}) will include the determination of the module structure of $J(K)$ whenever $G_K(p)$ is a finitely-generated free pro-$p$-group and $\Gal(K/F)$ is \emph{any} finite $p$-group (assuming either $\xi_p \in F$ or that $F$ is characteristic $p$).  In all cases, the computed modules have far fewer isomorphism types of summands than one expects generically. 

The very strong restrictions on the structures of this broad family of Galois modules represent some exciting developments about absolute $p$-Galois groups in the last 20 years.  Not only does this give some answers to how absolute $p$-Galois groups are special amongst pro-$p$ groups, but it has also been used to derive explicit group-theoretic properties that absolute $p$-Galois groups must satisfy.  Indeed, it is possible to provide module-theoretic parameterizations of certain families of $p$-groups as Galois groups over a given field (see \cite{MSS.auto,MS2,Schultz,Waterhouse}).  By combining these parameterizations with the computed module structures, enumerations (or bounds on enumerations) of Galois extensions have been deduced in \cite{BS,CMSHp3,Schultz}.  Additionally, unexpected connections between the appearance of certain groups as Galois groups over a given field (so-called \emph{automatic realization} results) are deduced in \cite{CMSHp3,MSS.auto,MS2,Schultz}.  Restrictions on absolute $p$-Galois groups in terms of generators and relations have also been derived from the computed module structure of cohomology groups in \cite{BLMS.prop.groups}. In \cite{MT-Advances}, ideas from Galois modules are used to provide an explicit proof of the Vanishing $3$-Massey conjecture, yet another result that constricts the structure of absolute $p$-Galois groups.

\subsection{The main object of our study}

The feasibility and applicability of previous module decompositions encourage us to make further investigations in the cases that might have been considered out of reach before the developments mentioned above.  The purpose of this paper is to give a module decomposition for a more refined notion of power classes than the $p$th power classes already considered.    

More specifically, let  $K/F$ be a cyclic Galois extension of degree
$p^n$, with Galois group $G=\Gal(K/F)= \langle \sigma \rangle$. For $m \in \N$, let
$R_m=\Z/p^m\Z$, and define the $R_mG$-module
\begin{equation*}
    J_m := K^\times/K^{\times p^m}. 
\end{equation*}
In this paper we determine the structure of  $J_m$. %; after we have given field-theoretic interpretations for certain invariants that appear in the module structure of $J_m$ in a subsequent paper, we provide the structure of $J_\infty$ as well.
Our work builds on the case $m=1$ from \cite{MSS}.  In the case of local fields, some results were announced in the 1970s (see \cite{R}).  Analogous results in the characteristic $p$ case have been computed in \cite{MSS.K.in.char.p}.  

Throughout this paper we will write $[\gamma]_m$ to indicate the class in $J_m$ represented by $\gamma \in K^\times$.  In a similar way, when $A \subset K^\times$ we will write $[A]_m$ to indicate the image of $A$ within $J_m$. 

Straightaway we should point out that the modular representation theory for $R_mG$ is very complex when $m>1$. When $n=1$ there is a classification of $R_mG$ modules due to Szekeres in \cite{Szekeres}, but when $n>1$ there are only partial results (see \cite{DrobotenkoEtAl,HannulaThesis,Hannula68,Thevenaz81}). Since a full classification of indecomposable $R_mG$-modules is not known, the reader would be forgiven for thinking that computing a module decomposition for $J_m$ would be impossible (or, at best, unspeakably grotesque).  To the contrary, we will ultimately show that for a given $K$ the module $J_m$ contains at most $n+1$ isomorphism classes of indecomposables, with at most one summand not a free $R_mG_i$-module for some quotient $G_i$ of $G$.  Further, if we fix $m$, $n \in \N$, there are only finitely many possible isomorphism classes for this non-free summand, and hence there are only finitely many possible isomorphism classes of summands of an arbitrary $J_m$.  In this sense, the decomposition we arrive at is striking in its simplicity.  That said, the decomposition is not without its challenges.  In particular, the non-free summand described above takes quite a bit of effort to verify as indecomposable; the relevant result is described in the lead up to the statement of Theorem \ref{th:main} and made concrete in Proposition \ref{pr:indecompII}, but the proof of the indecomposability is given in \cite{MSS2a}.

Before moving on, we observe that any $f \in R_mG$ can be represented uniquely in the form $\sum_{i=0}^{m-1}\sum_{j=0}^{p^n-1} a_{i,j} p^i (\sigma-1)^j$, where $a_{i,j} \in \{0,1,\cdots,p-1\}$.  (This expression is reminiscent of  $p$-adic expansions, though here the expansion moves in two directions.) For further details, see Section \ref{sec:module.basics}. In particular, for a given $[\gamma]_m \in J_m$, every element of the module $\langle [\gamma]_m\rangle$ can be written as a combination of the elements $\{[\gamma]_m^{p^i(\sigma-1)^j}: 0 \leq i \leq m-1, 0 \leq j \leq p^n-1\}$ with coefficients drawn from $\{0,1,\cdots,p-1\}$.   Sometimes it is useful to think of the terms in this module graphically, for which a grid of boxes --- with horizontal progression indicating successive powers of $\sigma-1$, and vertical progression indicating successive powers of $p$ --- can capture certain salient details. See Figure \ref{fig:box.drawing}. In this case, the top-left entry is the element which generates the cyclic module.  Building on the fact that $\sum_{j=0}^{p^n-1} \sigma^j = (\sigma-1)^{p^{n-1}}$ within $\mathbb{F}_p[G]$, the bottom right entry in this representation equals $[N_{K/F}(\gamma)]_m^{p^{m-1}}$.  

\begin{figure}[ht!]
\begin{tikzpicture}[scale=1.0]
\draw[thick,->] (1.5,.5) -- (3.5,.5);
\node at (2.5,1) {$\sigma-1$};
\draw[thick,->] (-.5,-.5) -- (-.5,-2.5);
\node at (-1,-1.5) {$p$};
\draw[step=1cm] (0,0) grid (5,-3);
\node at (0.5,-0.5) {$[\gamma]_3$};
\node at (0.5,-1.5) {$[\gamma]_3^p$};
\node at (1.5,-0.5) {$[\gamma]_3^{\sigma-1}$};
\filldraw[fill=orange!50] (4,-2) -- (5,-2) -- (5,-3) -- (4,-3) -- (4,-2);
\end{tikzpicture}
\caption{A graphical depiction of a cyclic module when $m=3$, $p=5$ and $n=1$.  The bottom right corner in this depiction --- which we've highlighted in orange --- is equivalent to $[N_{K/F}(\gamma)]_m^{p^{m-1}}$.}\label{fig:box.drawing}
\end{figure}
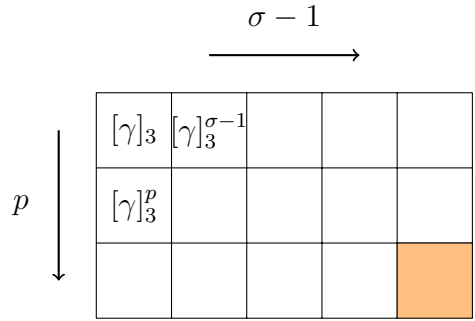
The downside to such a representation is that it can be difficult when $m>1$ to depict relations between terms of the form $[\gamma]_m^{p^i (\sigma-1)^j}$.  On the other hand, the simplest possible dependence --- when one of the terms is trivial --- can be depicted by excluding a given box. % Note that all boxes to the right and below an empty box are necessarily empty.  

We want to give the reader a sense of some of the critical ideas involved in moving from the known decomposition for $m=1$ to the cases where $m>1$.  Ultimately we'd like to produce some generating collection of elements in $J_m$ whose interrelations are understood. The tension is one that should be familiar to any linear algebra student: one must have enough generators to capture all of $J_m$, but not so many that an unexpected dependency arises.  

Fortunately, half of our work is actually quite simple.  If we have any collection $\{[\alpha_i]_1\}_{i \in \mathcal{I}}$ which span $J_1$, then the corresponding elements $\{[\alpha_i]_m\}_{i \in \mathcal{I}}$ must also span $J_m$.  For if we have some $[\gamma]_m \in J_m$, then we know that we can write $[\gamma]_1= \prod_i [\alpha_i]_1^{f_i}$ for some $f_i \in \mathbb{F}_p[G]$.  But this means that $\gamma = \left(\prod_i \alpha_i^{f_i}\right)\beta_1^p$.  Again using the fact that the $[\alpha_i]_1$ span $J_1$, we can find $g_i \in \mathbb{F}_p[G]$ and $\beta_2 \in K^\times$ so that $\beta_1 = \left(\prod_i \alpha_i^{g_i}\right) \beta_2^p$.  Combining this with the previous equation gives $\gamma  = \left(\prod_i \alpha_i^{f_i+pg_i}\right)\beta_2^{p^2}.$  Continuing in this fashion, we ultimately reach our desired result.

Resolving the issue of independence is not always  simple.  We will now explain where one can run into complications.  Suppose that $[\gamma]_{m-1} \in J_{m-1}$ generates a free $R_{m-1}G$-module.  Will $[\gamma]_m \in J_m$ also generates a free $R_mG$-module?  To study this question, it will be useful to know that an element $[\hat \gamma]_m \in J_m$ generates a free module if and only if $[N_{K/F}(\hat \gamma)]_m^{p^{m-1}} \neq [1]_m$ (see Lemma \ref{le:idealrmgiII}).  With this in mind, we have that $\langle [\gamma]_m \rangle$ fails to be free precisely when there is some $\beta \in K^\times$ with $N_{K/F}(\gamma)^{p^{m-1}} = \beta^{p^m}$.  If this occurs, then we can extract $p$th roots on both sides.  Note then that for some $p$th root of unity $\omega \in K^\times$ we get $N_{K/F}(\gamma)^{p^{m-2}} = \omega \beta^{p^{m-1}}$.  Herein lies the rub: though we know that $[N_{K/F}(\gamma)]_{m-1}^{p^{m-2}} \neq [1]_{m-1}$, this doesn't preclude the possibility that $[N_{K/F}(\gamma)]^{p^{m-2}}_{m-1} \in \langle [\omega]_{m-1} \rangle$.  See Figure \ref{fig:box.drawing.free.to.nonfree}.
%\begin{center}
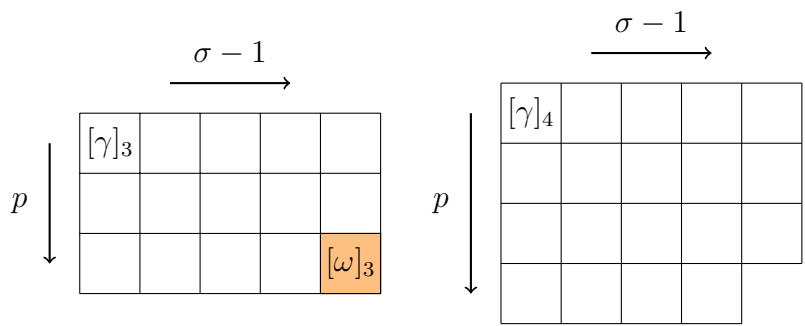
\begin{figure}[ht!]
\begin{tikzpicture}[scale=0.8]
\draw[thick,->] (1.5,.5) -- (3.5,.5);
\node at (2.5,1) {$\sigma-1$};
\draw[thick,->] (-.5,-.5) -- (-.5,-2.5);
\node at (-1,-1.5) {$p$};
\filldraw[fill=orange!50] (4,-2) -- (5,-2) -- (5,-3) -- (4,-3) -- (4,-2);
\draw[step=1cm] (0,0) grid (5,-3);
\node at (0.5,-0.5) {$[\gamma]_3$};
\node at (4.5,-2.5) {$[\omega]_3$};
\draw[thick,->] (8.5,1) -- (10.5,1);
\node at (9.5,1.5) {$\sigma-1$};
\draw[thick,->] (6.5,0) -- (6.5,-3);
\node at (6,-1.5) {$p$};
\draw (7,0.5) -- (12,0.5);
\draw (7,-0.5) -- (12,-0.5);
\draw (7,-1.5) -- (12,-1.5);
\draw (7,-2.5) -- (12,-2.5);
\draw (7,-3.5) -- (11,-3.5);
\draw (7,0.5) -- (7,-3.5);
\draw (8,0.5) -- (8,-3.5);
\draw (9,0.5) -- (9,-3.5);
\draw (10,0.5) -- (10,-3.5);
\draw (11,0.5) -- (11,-3.5);
\draw (12,0.5) -- (12,-2.5);
\node at (7.5,0) {$[\gamma]_4$};
\end{tikzpicture}
\caption{$\langle [\gamma]_{m-1}\rangle$ can be free without $\langle [\gamma]_m\rangle$ being free if $[N_{K/F}(\gamma)]_{m-1}^{p^{m-1}}$ is a nontrivial $p$th root of unity.}\label{fig:box.drawing.free.to.nonfree}
\end{figure}
%\end{center}

There are a few takeaways from this example.  First, if one happens to be in the happy situation where the entire module $J_{m-1}$ is free, and the field simply doesn't have any non-trivial $p$th roots of unity, then the sketch above provides the blueprint for extending the freeness of $J_{m-1}$ to $J_m$.  Fortunately, there are conditions on a field $K$ that ensure $J_1$ is (essentially) free, and they include the absence of non-trivial roots of unity.  This provides the easiest case to analyze, and this case is detailed in Theorem \ref{th:noxip}.  

Generically, though, we \emph{do} expect $J_1$ to have non-free summands, as well as a variety of $p$th roots of unity (and, indeed, higher $p$th power roots of unity as well).  The challenge becomes twofold.  On the one hand, free cyclic modules are  easier to work with, and so we'd like to have as many free cyclic summands as we can; this requires work.   On the other hand, the sketch above shows the root\footnote{no pun intended} of many complications as we move from $J_{m-1}$ to $J_m$: the presence of a primitive $p$th roots of unity in equations modulo $K^{\times p^{m-1}}$.  Therefore, we need to be aware of where a primitive $p$th root of unity sits.  If it can be suitably quarantined in its own indecomposable summand, it will be easier to prevent its interference in the rest of the module.

\subsection{Main Results}
As we mentioned above, we will ultimately show that each $J_m$ contains at most $n+1$ isomorphism classes of indecomposables, with at most one that is not a free $R_mG_i$-module for some quotient $G_i$ of $G$.  If $J_m$ contains such a non-free summand, we call it an exceptional summand of $J_m$, and we determine its structure in terms of the arithmetic of $K/F$.  Such a summand appears only when $F$ contains a primitive $p$th root of unity $\xi_p$, but does not contain all of the $p$-power roots of unity in $F^{\text{sep}}$.  If we let $\nu$ be chosen so that $\xi_{p^\nu}$ is the maximal $p$-power root of unity in $K$, then the isomorphism type of the exceptional module is determined by $\nu$ in the case $p=2$, $n=1$ and $-1\not\in N_{K/F}(K^\times)$.  Otherwise, the structure of an exceptional summand is determined by a vector $\mathbf{a}\in \{-\infty,0,\dots,n\}^m$ and a natural number $d\in \N$.

Both the elements of $\mathbf{a}$ and the number $d$ are determined by examining so-called ``minimal norm pairs" in $K^\times$ (see section \ref{se:np1}). In \cite{MSS2c} we show that the elements of $\mathbf{a}$ are determined by the
presence of $p$-power roots of unity as norms from $K$ to subfields
of $K/F$, and hence are unique.  In that paper we show $d$ is uniquely-defined only modulo $p^{\min\{m,\nu\}}$, and that it is the cyclotomic character.

To state our results more precisely, we first develop some notation which will be useful for the balance of the paper. For $i\in \N$ set $U_i=1+p^i\Z$ and additionally set $U_\infty=\{1\}$. We adopt several conventions: $p^{-\infty}=0$, $\sigma^{p^{-\infty}}=0$, and $\{0\}$ is a free module over any ring. For $l\in \N$, $d\in U_1$, and a vector $\mathbf{a}\in \{-\infty, 0, \dots,n\}^l$ of length $l$, define the $\Z_p G$-module $X_{\mathbf{a},d}$ by
\begin{align*}
    \left\langle y,x_0,\cdots,x_{l-1} :
    (\sigma-d)y = \sum_{i=0}^{l-1} p^i x_i; \sigma^{p^{a_i}}x_i=x_i \text{ for all }0 \leq i \leq l-1\right\rangle.
\end{align*}
(Note that when $a_i=-\infty$, our convention that $\sigma^{p^{-\infty}} = 0$ tells us that the corresponding $x_i$ equals $0$.)  Then for $m\in \N$ define
\begin{equation*}
    X_{\mathbf{a},d,m} = X_{\mathbf{a},d}/p^mX_{\mathbf{a},d}.
\end{equation*}
It is shown in \cite[Prop.~5.1]{MSS2a} that under certain conditions on $\mathbf{a}$ and $d$ (listed in Proposition \ref{pr:indecompII}), the vector $\mathbf{a}$ is a module-theoretic invariant of $X_{\mathbf{a},d,m}$: if $(\mathbf{a}_1,d_1)$ and $(\mathbf{a}_2,d_2)$ both satisfy the conditions of Proposition \ref{pr:indecompII}, and if $\mathbf{a}_1 \neq \mathbf{a_2}$, then $X_{\mathbf{a}_1,d_1,m} \not\simeq X_{\mathbf{a}_2,d_2,m}$.  On the other hand, even under the conditions of Proposition \ref{pr:indecompII} the quantity $d$ is not necessarily a module theoretic invariant (see \cite[Th.~1]{MSS2c}).

For $i\in \{0, \dots, n\}$, let $K_i$ denote the intermediate field of $K/F$ of degree $p^i$ over $F$. For convenience it will often be useful to write $K_{-\infty}^\times = \{1\}$, and furthermore to interpret $K_{i-1}^\times$ as $K_{-\infty}^\times$ when $i=0$.  We abbreviate $\Gal(K_i/F)$ by $G_i$.  Moreover, define
\begin{align*}
    e_i(K/F) &= \dim_{\Fp} N_{K_i/F}(K_i^\times)F^{\times p}/
    N_{K_{i+1}/F}(K_{i+1}^\times )F^{\times p}, \quad 0\le i<n\\
    e_n(K/F) &= \dim_{\Fp} N_{K/F}(K^\times)F^{\times p}/F^{\times p}.
\end{align*}

We can now state the primary results.

\begin{theorem}\label{th:main}
    Let $m \in \N$, and suppose that $\xi_p\in F$.  If $p = 2$ and $n=1$ suppose that
    $-1 \in N_{K/F}(K^\times)$ as well.  Then
    \begin{equation*}
        J_m \simeq X \oplus \bigoplus_{i=0}^n Y_i,
    \end{equation*}
    where for some $\mathbf{a}\in \{-\infty,0,\dots,n\}^m$ and $d\in
    U_1$,
    \begin{enumerate}
        \item $X$ is an indecomposable $R_mG$-module isomorphic to
        $X_{\mathbf{a},d,m}$; and
        \item for each $i$, $Y_i$ is a free $R_m G_i$-module
        of rank $e(i,m)$, where
        \begin{align*}
            e(i,m) + \vert\{a_j=i\ :\ 0\le j<m\}\vert &= e_i(K/F),
            \quad 0\le i<n \\
            e(n,m) + 1 + \vert\{a_j=n\ :\ 0\le j<m\}\vert &=
            e_n(K/F).
        \end{align*}
    \end{enumerate}
\end{theorem}

%\begin{theorem}\label{th:decomposition.for.J.infinity}
%    Continue with the hypotheses of Theorem \ref{th:main}. Then there exists $\mathbf{a} \in \{-\infty,0,\dots,n\}^\infty$ and $d \in U_1$ so that
%    \begin{equation*}
%        J_\infty \simeq X \oplus \bigoplus_{i=0}^n Y_i,
%    \end{equation*}
%    where
%    \begin{enumerate}
%        \item $X$ is a finitely generated indecomposable $\Z_p G$-module isomorphic to $X_{\mathbf{a},d}$
%        \item for each $i$, $Y_i$ is a free $\Z_p G_i$-module of rank $e(i,\nu+1)$ if $\nu<\infty$ and $e_i(K/F)$ if $\nu=\infty$.
%    \end{enumerate}
%\end{theorem}

%The remaining case presents a few special technical considerations, and hence varies slightly from the results above.

\begin{theorem}\label{th:p2n1m1notnorm}
    Suppose that $p=2$, $\xi_2\in F$, $n=1$, and $-1 \notin N_{K/F}(K^\times)$. Then
    \begin{equation*}
        J_\infty \simeq \begin{cases} Z \oplus Y_0 \oplus Y_1,
        &\nu<\infty\\
        Y_0\oplus Y_1, &\nu=\infty
        \end{cases}
    \end{equation*}
    where: $Y_0$ is a trivial $\Z_2 G$-module which is a free $\Z_2$-module with rank $e(0)=e_0(K/F)-1$;  $Y_1$ is a free $\Z_2 G$-module of rank $e(1)=e_1(K/F)-1$ when $\nu<\infty$ and
    $e(1)=e_1(K/F)$ when $\nu=\infty$; and, when $\nu<\infty$,
    $Z\simeq \Z_2G/\langle 2^{\nu}(\sigma-1)\rangle$ is indecomposable.

    Hence for $m\in \N$,
    \begin{equation*}
        J_m \simeq \begin{cases} (Z/2^m) \oplus
        (Y_0/2^m) \oplus (Y_1/2^m), &\nu<\infty\\
        (Y_0/2^m) \oplus (Y_1/2^m), &\nu=\infty
        \end{cases}
    \end{equation*}
    where $Z/2^mZ$ is indecomposable, $Y_0$ is a trivial $R_m G$-module which is a free $R_m$-module with rank $e(0)$, and $Y_1$ is a free $R_m G$-module of rank $e(1)$.
\end{theorem}

\begin{theorem}\label{th:noxip}
    Suppose that $F$ contains no primitive $p$th root of unity. Then
    \begin{equation*}
        J_\infty \simeq \bigoplus_{i=0}^n Y_i,
    \end{equation*}
    where for each $i$, $Y_i$ is a free $\Z_p G_i$-module of rank $e_i(K/F)$.

    Hence for $m\in \N$,
    \begin{equation*}
        J_m \simeq \bigoplus_{i=0}^n \ (Y_i/p^m),
    \end{equation*}
    where for each $i$, $Y_i/p^m$ is a free $R_m G_i$-module of rank $e_i(K/F)$.
\end{theorem}

The decompositions we find in the preceding theorems follow a trend that has been emerging in module decompositions related to Galois cohomology in the last two decades: the decomposition itself consists mostly of ``free" summands, with only a small ``exceptional" summand that encodes certain important arithmetic information.  

It might seem somewhat prosaic to say that the trend simply continues in this case. One must remember, however, that the modular representation theory in this setting is vastly more complex than module decompositions that have appeared in the literature.  Given the chaotic nature of modules in this particular setting, \emph{a priori} it would have been surprising to know that a module decomposition was possible it all.  The fact that the decomposition \emph{is} possible --- and that the summands are limited to such a thin set of possibilities (including a new indecomposable type that was born out of an analysis of arithmetic conditions related to norms) --- makes the results in this paper a considerable advertisement for the continued analysis of modules related to Galois cohomology.  Outside of the direct applications of these results, we hope this paper inspires further explorations of objects such as these, even in cases where the modular representation theory makes their decomposition seemingly intractable.

Finally, we point out that the current suite of results on Galois modules (used as a tool for analyzing absolute Galois groups) has some natural interconnections with other research agendas in and around algebra.  For example, one extremely powerful technique in Galois theory has been that of ``patching."  Harbater used this as a key tool in proving that every finite group is Galois group over a regular extension of $\mathbb{Q}_p(t)$ (see \cite{Harbater}, as well as \cite{Raynaud}).  For a basic reference, see \cite[Ch.~8]{SerreTopics}; a recent survey article on topics related to this result is \cite{Harbater-Obus-Pries-Stevenson}.  Patching has proved to be an exceptionally versatile and powerful tool, and has been applied to the study torsors, central simple algebras, quadratic forms, and absolute Galois groups (see \cite{Harbater-et-al,Harbater-Hartmann,Harbater-Hartmann-Krashen-torsors,Harbater-Hartmann-Krashen-quadratic-forms,Jarden.Patching}).  The patching method could be a reasonable tool to reach for when constructing suitable Galois modules, where one could conceivably begin with smaller modules as building blocks for constructing larger modules.  This would be a very powerful new tool to add to the arsenal of techniques in the investigation of Galois modules (and therefore absolute Galois groups). 

As another example, we point out that these module investigations might have interesting connections to anabelian geometry.  Progress has been made in studying the anabelian features of Galois groups using surprisingly small groups; see \cite{BogRovTsc,BogTsc1,BogTsc2,BogTsc3,BogTsc4,BogTsc5,BogTsc6,BogTsc7,BogTsc8,CEM,EfratMinac,MSpira}. Interestingly, anabelian geometry also has deep connections to rational points on curves. (The interested reader can consult \cite{Grothendieck} for the origin of the section conjecture, as well as \cite{JardenPetersen,Wickelgren.3.nilpotent,Wickelgren.n.nilpotent,Wickelgren.2.nilpotent,Wickelgren.massey} for some recent work in this area.)  One can see examples of Galois modules used in service of understanding arithmetic geometry problems related to Fermat curves; see \cite{Anderson,DavisPriesWickelgren,DavisPriesStojanoskaWickelgren.Fermat,DavisPriesStojanoskaWickelgren.relativeFermat}. Given the strong connection between Galois modules and absolute Galois groups, a determination of Galois module structures from Galois groups would fit very nicely in the birational anabelian program.

\subsection{Outline and notational conventions} 
In section \ref{sec:module.basics} we recall a few basic module-theoretic facts from \cite{MSS2a}.  We then offer proofs of Theorems \ref{th:p2n1m1notnorm} and \ref{th:noxip}, since they can be resolved with relative ease.  The rest of the paper focuses on the case where $\xi_p \in F$, and assumes that $-1 \in N_{K/F}(K^\times)$ when $p=2$ and $n=1$.  In section \ref{se:np1} we introduce norm pairs and exceptional elements, and we explore their connection to a notion of exceptionality defined in \cite{MSS}. In section \ref{se:np2} certain fundamental inequalities are established for minimal norm pairs, and in section \ref{se:excelts} we investigate the nature of elements representing minimal norm pairs when considered in $J_1$.  
In section \ref{sec:exceptional.modules} we show that one can construct an exceptional summand of $J_m$ with elements representing a minimal norm pair, and that this summand is indecomposable.  With all of this machinery in hand, we give a proof for Theorem \ref{th:main} in section \ref{sec:main.proofs}.

Up to this point we have been writing $J_m$ multiplicatively, which is a natural enough choice considering it is a quotient of the multiplicative group $K^\times$.  On the other hand, this forces the $R_mG$-action to be exponential, and our expressions will often be complex enough that recording things exponentially makes for difficult reading.  Hence, from this point forward we will write the module $J_m$ additively, so that the action of $R_mG$ is multiplicative.  Starting in section \ref{se:np1} we will adopt additive notation when discussing the group $K^\times$ as well; in particular we will write $0$ for the element $1 \in K^\times$, and $\frac{1}{p}\gamma$ to indicate $\root{p}\of{\gamma}$ for $\gamma \in K^{\times p}$. (We avoid using additive notation for $K^\times$ in section \ref{sec:proofs.in.non.generic.cases} because there is one place where we'll need to consider the sum of two elements from $K^\times$.)

\subsection{Acknowledgements}

We gratefully acknowledge discussions and collaborations with our friends and colleagues L.~Bary-Soroker, S.~Chebolu, F.~Chemotti, I.~Efrat, A.~Eimer, J.~G\"{a}rtner, S.~Gille, P.~Guillot, L.~Heller, D.~Hoffmann, J.~Labute, C.~Maire, D.~Neftin, N.D.~Tan, A.~Topaz, R.~Vakil and K.~Wickelgren which have influenced our work in this and related papers. We would also like to thank the anonymous referee who provided valuable feedback that improved the quality and exposition of this manuscript.

\section{Some module-theoretic basics}\label{sec:module.basics}

The following results were established in \cite{MSS2a} and will be useful in our current investigation.  The first two are arithmetic in nature, while the remainder give important information about $R_mG$ modules.

\begin{lemma}[{\cite[Lem.~2.1]{MSS2a}}]\label{le:upowerII}
    Suppose that $i\in \N$, and $d\in U_i$,
    and $j\ge 0$.  Then
    \begin{equation*}
        \begin{cases}
            d^{p^j}\in U_{i+j},
            &p>2 \text{\ or\ } i>1 \text{\ or\ } j=0\\
            d^{p^j}=1, &p=2, \ i=1, \ j>0, \ d=-1\\
            d^{p^j}\in U_{v+j},
            &p=2, \ i=1, \ j>0, \ d\in -U_v
        \end{cases}
    \end{equation*}
	When additionally $d \not\in U_{i+1}$, then we have $d^{p^j} \not\in U_{i+j+1}$ in the first case; likewise if $d \not\in -U_{v+1}$ in the last case, then $d^{p^j} \not\in U_{v+j+1}$.
\end{lemma}

Of central importance in much of this paper are the family of operators from $\Z G$ defined by $$P(i,j) = \sum_{k=0}^{p^{i-j}-1}\left(\sigma^{p^j}\right)^k \quad \text{ for } 0 \leq j \leq i \leq n.$$  Note that for $\gamma \in K_i^\times$ we have $\gamma^{P(i,j)} = N_{K_i/K_j}(\gamma)$, and furthermore that $$P(i,j) = (\sigma-1)^{p^i-p^j} \quad \pmod{p\Z G}.$$ 

In the following result, we analyze the image of a polynomial $P(i,j)$ under a certain evaluation map.  Specifically, for a number $d \in U_1$ the function $\phi_d:\Z G \to \Z$ is induced by $\sigma^t \mapsto d^t$ for  $0 \leq t<p^n$.  (In the case where $d^{p^n} = 1 \pmod{p^m}$, note that this induces a map $\phi_d^{(m)}:R_mG \to R_m$ which has kernel $\langle \sigma-d\rangle$.)

\begin{lemma}[{\cite[Lem.~2.3]{MSS2a}}]\label{le:phiII}
    For $0\le j\leq i\le n$
    \begin{equation*}
       \phi_d(P(i,j)) = 0 \pmod{p^{i-j}}.
    \end{equation*}

    More precisely,
    \begin{enumerate}
        \item\label{it:phi2II} If $p>2$, $d\in U_2$, or $j>0$, then
          for $0\le j<i$
        \begin{equation*}
            \phi_d(P(i,j)) = p^{i-j} \pmod{p^{i-j+1}}.
        \end{equation*}
        \item\label{it:phidm1II} If $p=2$, $d=-1$, $j=0$, and $i>0$, then
        \begin{equation*}
          \phi_d(P(i,0)) = 0.
        \end{equation*}
        \item\label{it:phi3II} If $p=2$, $d\in -U_v\setminus -U_{v+1}$ for $v\ge 2$, $j=0$, and
          $i>0$, then
        \begin{equation*}
            \phi_d(P(i,0)) = 2^{i+v-1} \pmod{2^{i+v}}.
        \end{equation*}
    \end{enumerate}
\end{lemma}

The next few lemmas give some basic properties of $R_mG$-modules.  Before we state them, note that for an $R_mG$-module $M$ there is an associated trivial $R_mG$-submodule $M^\star$ given by
\begin{equation*}
    M^\star := \ann_{M} \langle \sigma-1,p\rangle = \ann_M (\sigma-1) \cap \ann_M p.
\end{equation*}
Since $M^\star$ depends on $\ann_M (\sigma-1)$ and $\ann_M p$, we first describe these annihilators in the case where $M = R_mG_i$.

\begin{lemma}[{\cite[Lem.~2.5]{MSS2a}}]\label{le:kerbasicII}
    Suppose $m\in \N$.
    For $0\le i\le n$ and $0\le k\le m$,
    \begin{equation*}
        \ann_{R_mG_i} p^k = \langle p^{m-k}\rangle.
    \end{equation*}

    For $0\le j<i\le n$ and $0\le k\le m$,
    \begin{equation*}
        \ann_{R_mG_i} p^k(\sigma^{p^j}-1)= \langle P(i,j),
        p^{m-k}\rangle.
    \end{equation*}
\end{lemma}

The previous result lets us conclude the following.

\begin{lemma}[{\cite[Lem.~3.4]{MSS2a}}]\label{le:starrmgiII}
    For $m\in \N$ and $0\le i\le n$,
    \begin{equation*}
        (R_mG_i)^\star = \langle p^{m-1}P(i,0)\rangle = \langle p^{m-1}(\sigma-1)^{p^i-1}\rangle \neq \{0\}.
    \end{equation*}
\end{lemma}

\begin{lemma}[{\cite[Lem.~3.5]{MSS2a}}]\label{le:idealrmgiII}
    For $m\in \N$ and $0\le i\le n$, every nonzero ideal of $R_mG_i$
    contains
    \begin{equation*}
        p^{m-1}(\sigma-1)^{p^i-1}
        = p^{m-1}P(i,0).
    \end{equation*}
\end{lemma}

The importance of $M^\star$ in understanding a general $R_mG$-module $M$ is outlined in the following results.

\begin{lemma}[{\cite[Lem.~3.2]{MSS2a}}]\label{le:starzeroII}
    Let $M$ be a nonzero $R_mG$-module.  Then $M^\star\neq \{0\}$.
\end{lemma}

\begin{lemma}[{\cite[Lem.~3.3]{MSS2a}}]\label{le:exclII}
    Let $0\le i\le n$, and let $M_1$ and $M_2$ be $R_mG_i$-submodules contained in a common $R_mG_i$-module.  If $M_1^\star\cap M_2^\star = \{0\}$ then $M_1+M_2 = M_1\oplus M_2$.
\end{lemma}

Finally, we state a result which gives conditions (on $\mathbf{a}$ and $d$) under which the module $X_{\mathbf{a},d,m}$ is indecomposable.

\begin{proposition}[{\cite[Th.~1]{MSS2a}}]\label{pr:indecompII}
    Suppose $m\in \N$,
    \begin{equation*}
        \mathbf{a}=(a_0,\dots,a_{m-1}) \in
        \{-\infty,0,\dots,n\}^{m}
    \end{equation*}
    and $d\in \Z$ satisfy
    \begin{enumerate}
        \renewcommand{\theenumi}{\Roman{enumi}}
        \item\label{it:exc.mod.indecom.condition...d.in.U1II} $d\in U_1$
        \item\label{it:exc.mod.indecom.condition...power.of.dII} $d^{p^{a_i}}\in U_{i+1}$ for all $0\le i<m$
        \item\label{it:exc.mod.indecom.condition...a0.boundsII} if $p=2$ and $n=1$, then $a_0 = -\infty$
        \item\label{it:exc.mod.indecom.condition...ai.vs.aj.inequalityII}
          $a_i+j<a_{i+j}$ for all $0\le i<m$ with  $1\le j<
          m-i$ and $a_{i+j}\neq -\infty$, except if $p=2$, $d\not\in
          U_2$, $i=0$, and $a_i=0$, in which case $a_j\neq 0$ for all
          $1\le j<m$
        \item\label{it:exc.mod.indecom.condition...ai.vs.v.inequalityII}
          If $p=2$, $m\ge 2$, $d\in -U_v\setminus -U_{v+1}$ for
          some $v\ge 2$, and $a_0=0$, then $a_{i}>i-(v-1)$ for
          all $v\le i<m$ and $a_i\neq -\infty$.
    \end{enumerate}
    Then $X_{\mathbf{a},d,m}$ is an indecomposable $R_mG$-module.
\end{proposition}

\section{Proofs of Theorems~{{\ref{th:p2n1m1notnorm}}} and {{\ref{th:noxip}}}}\label{sec:proofs.in.non.generic.cases}

The proofs of Theorems~\ref{th:p2n1m1notnorm} and \ref{th:noxip} follow relatively quickly from the module structure of $J_1$, because in these cases it's straightforward to account for relations which appear in $J_m$.  We will begin with Theorem \ref{th:noxip} since the proof is simplest in this case.

%Before we proceed with the proofs of these results, we first make a few comments on notation.  We maintain the assumptions of section~\ref{se:np1}, but we begin to write the operation of $K^\times$ additively: we write $0$ for the identity; we use multiplication by $p^l$ to denote exponentiation by $p^l$; and we use multiplication by $1/p^l$ to denote $p^l$-th roots in $K^{\text{sep}\ \times}$. When we write $(p^l)$ to the right of equations we mean the equality modulo $p^l$-th powers in $K^\times$. 

%We write classes in $J_i=K^{\times}/p^iK^\times$ as $\gamma\ (p^i)$ for $\gamma\in K^\times$.  Moreover, to simplify notation we write $\langle \gamma\rangle\ (p^i)$ in place of $\langle \gamma (p^i)\rangle$ to denote the $R_i G$-submodule of $J_i$ generated by the class of $\gamma$, and for $A\subset K^\times$, the notation $A (p^i)$ to denote the image of $A$ in $J_i$.

\begin{proof}[Proof of Theorem~\ref{th:noxip}]
We assume that $F$ contains no primitive $p$th root of unity.  By \cite[Th.~1]{MSS}, we can select elements $\mathcal{K} \subset K^\times$ so that $$J_1 = \bigoplus_{\gamma \in \mathcal{K}} \langle [\gamma]_1\rangle,$$ where for any $\gamma \in \mathcal{K}\cap (K_i^\times\setminus K_{i-1}^\times)$ we have $\langle [\gamma]_1\rangle \simeq \F_pG_i$.  If we let $e_i(\K) = |\K\cap (K_i^\times \setminus K_{i-1}^\times)|$ for $0 \leq i \leq n$, then \cite[Cor.~1]{MSS} gives $e_i(\K) = e_i(K/F)$.

We first show that for each $m\in \N$, the set $\left\{\langle [\gamma]_m\rangle: \gamma\in \K\right\}$ is an independent set, and that for each $\gamma\in \K$ we have $\langle [\gamma]_m\rangle\simeq R_mG_{e(\gamma)}$
where
\begin{equation*}
    e(\gamma)=\min \{i\in\{-\infty, 0,\dots,n\} :
    \gamma\in K_i^\times\}.
\end{equation*}

We proceed by induction.  The case $m=1$  has been handled above. Suppose that for $m \geq 2$ the set of cyclic modules $\left\{\langle [\gamma]_{m-1}\rangle:\gamma\in \K\right\}$ is an independent set, and that for each $\gamma\in \K$ we have $\langle [\gamma]_{m-1}\rangle \simeq R_{m-1}G_{e(\gamma)}$.  We show that the statements hold for $m$ as well.

First consider $\langle [\gamma]_m\rangle$ for $\gamma\in \K$, and --- for simplicity --- write $e=e(\gamma)$. Since $\gamma\in K_{e}^\times$, it follows that $\langle [\gamma]_m \rangle$ is an $R_mG_e$-module.  Since $\langle [\gamma]_{m-1}\rangle$ is a free $R_{m-1}G_e$-module, Lemma~\ref{le:idealrmgiII} gives,
\begin{equation*}
    p^{m-2}P(e,0)[\gamma]_{m-1}\neq [0]_{m-1}.
\end{equation*}
Now suppose $\langle [\gamma]_m \rangle$ is not free as an $R_mG_e$-module. Then by Lemma~\ref{le:idealrmgiII},
\begin{equation*}
    p^{m-1}P(e,0)[\gamma]_m = [0]_m.
\end{equation*}
Hence for some $\beta \in K^\times$ we have
\begin{equation*}
    (\gamma)^{p^{m-1}P(e,0)} = \beta^{p^m}.
\end{equation*}
Observe that because $\xi_p\not\in F^\times$ and $[K:F]=p^n$, we have $\xi_p\notin K^\times$.  Therefore $p$th roots of $p$th powers are uniquely defined in $K^\times$ and we deduce
\begin{equation*}
    (\gamma)^{p^{m-2}P(e,0)} = \beta^{p^{m-1}}.
\end{equation*}
Therefore
\begin{equation*}
    p^{m-2}P(e,0)[\gamma]_{m-1} =  [0]_{m-1},
\end{equation*}
a contradiction.

Now we claim that $\left\{\langle [\gamma]_m \rangle: \gamma \in \K\right\}$ is independent.  By induction the set $\left\{\langle [\gamma]_{m-1}|\rangle:\gamma\in \K\right\}$ is independent. We proceed by induction on the number of modules $\langle [\gamma]_m \rangle$ that we already know to be independent. The base case is clear. Now consider $\langle [\gamma]_m \rangle$ and $\oplus \langle [\gamma_i]_m \rangle$. By Lemma~\ref{le:exclII} it suffices to show that $\langle [\gamma]_m \rangle^\star \cap \oplus \langle [\gamma_i]_m \rangle^\star = \{0\}$.

Suppose the contrary.  We have that each $\langle [\gamma_i]_m\rangle$ is a free
$R_mG_{e(\gamma_i)}$-module.  By Lemma~\ref{le:starrmgiII}
\begin{equation*}
    \langle[\gamma_i]_m\rangle^\star = \langle p^{m-1}P(e(\gamma_i),0)\rangle,
\end{equation*}
and this module is isomorphic to $\Fp$.  Hence we must have a relation
\begin{equation*}
    c_\gamma p^{m-1}P(e,0)[\gamma]_m + \sum c_{\gamma_i}
    p^{m-1}P(e(\gamma_i),0)[\gamma_i]_m = [0]_m,
\end{equation*}
with $c_\gamma$ and $c_{\gamma_i}\in \Z$ and not all divisible by $p$.  Taking $p$th roots, we obtain
\begin{equation*}
    c_\gamma p^{m-2}P(e,0)[\gamma]_{m-1} + \sum c_{\gamma_i}
    p^{m-2}P(e(\gamma_i),0)[\gamma_i]_{m-1} = [0]_{m-1}\quad,
\end{equation*}
contradicting the fact that the modules $\left\{\langle [\gamma]_{m-1}\rangle: \gamma \in \K\right\}$ are independent.

Now for each $i$, let $Y_i$ be the span of $\langle [\gamma]_m\rangle$ for $\gamma\in \K \cap (K_i^\times \setminus K_{i-1}^\times)$. By definition the rank of $Y_i$ is $e_i(\K)$, and we have already argued that $e_i(\K) = e_i(K/F)$.  The inverse limit calculation follows naturally.
\end{proof}

The proof of Theorem~\ref{th:p2n1m1notnorm} relies on variations of the same theme, though the fact that $p=2$, $n=1$ and $-1\not\in N_{K/F}(K^\times)$ will present some technical issues.  

\begin{proof}[Proof of Theorem~\ref{th:p2n1m1notnorm}]
Assume  $p=2$, $\xi_2\in F$, $n=1$, and $-1\not\in N_{K/F}(K^\times)$.  Note that if $\xi_4 \in F$ then $N_{K/F}(\xi_4)=-1$, contrary to hypothesis.  Hence $\xi_4 \not\in F$.

First suppose that $\nu=\infty$; i.e., that $K$ contains all $2$-power roots of unity.  Then we modify the proof
of Theorem~\ref{th:noxip} as follows.  There we use the
fact that $p$th roots of $p$th powers are uniquely defined
(since $\xi_p\not\in K^\times$) to deduce that, for $a, b\in K^\times$ and $m\ge 2$,
\begin{equation*}
    a^{p^{m-1}}=b^{p^m} \implies a^{p^{m-2}}=b^{p^{m-1}}
\end{equation*}
and therefore that
\begin{equation*}
    p^{m-1} [a]_m = [0]_m  \implies p^{m-2} [a]_{m-1} = [0]_{m-1}.
\end{equation*}
We show that this last result still holds in our case.  Suppose that $a^{2^{m-1}}=b^{2^m}$, for some $a, b\in K^\times$, and $m\ge 2$.  Taking square roots, we have $ a^{2^{m-2}} (-1)^c = b^{2^{m-1}}$ for some $c\in \Z$. But since $\nu=\infty$, $(-1)^c=d^{2^{m-1}}$ for some $d\in K^\times$.  Hence $2^{m-1}[a]_m=[0]_m$ implies $2^{m-2}[a]_{m-1}=[0]_{m-1}$, as desired.  The rest of the proof carries over without modification.

Now we suppose that $\nu<\infty$.  Since $\xi_2\in F$ we have $\nu\ge 1$.  If $\nu=1$ set $\lambda=\sqrt{a}$, where $K = F(\root\of{a})$.  If instead $\nu> 1$ then $K=F(\xi_4)$ and we set $f_0=(\xi_{2^{\nu}}+\xi_{2^{\nu}}^\sigma)/2\in F$, $f_1=(\xi_{2^{\nu}}-\xi_{2^{\nu}}^\sigma)/(2\xi_4)\in F$, and $\lambda = f_1+(1-f_0)\xi_4$.  Note that $N_{K/F}(\xi_{2^\nu})$ is a $2^{\nu-1}$ root of unity in $F$, and so $N_{K/F}(\xi_{2^\nu}) = \pm 1$.  But since $-1 \not\in N_{K/F}(K^\times)$ by assumption, we must have $N_{K/F}(\xi_{2^\nu}) = 1$.  In particular, this means that $\xi_{2^\nu}^\sigma = \xi_{2^\nu}^{-1}$.  Observe that $\xi_{2^{\nu}}=f_0+f_1\xi_4$, that $\xi_{2^{\nu}}^{-1}=f_0-f_1\xi_4$, and that $f_0^2+f_1^2=1$.  In either case, we calculate that $\lambda^{\sigma-1} = \xi_{2^{\nu}}^{-1}$.  Observe that $[-1]_m\in \langle \lambda\rangle$ for all $m\in \N$.

Now if $[\xi_{2^{\nu}}]_1^{-1}= [0]_1$ then there exists a primitive $2^{\nu+1}$th root of unity in $K$, contradicting the definition of $\nu$.  Hence $(\sigma-1)[\lambda]_1\neq[0]_1$, and so the $\Ft G$-module $\langle [\lambda]_1 \rangle$ is isomorphic to $\Ft G$.  The proof of \cite[Th.~1]{MSS} (specifically, the first paragraph of \cite[Prop.~3]{MSS}) tells us that we can select $\K \subset K^\times$ with $\lambda \in \K$ and so that $$J_1 = \bigoplus_{\gamma \in \K} \langle[\gamma]_1\rangle,$$ where for any $\gamma \not\in K_0^\times$ we have $\langle[\gamma]_1\rangle$ is a free $\Ft G$-module, and for any $\gamma \in K_0^\times$ we have $\langle [\gamma]_1\rangle$ is a trivial $\Ft G$-module.  

We prove the theorem by showing that for each $m\in \N$ we have:  the set $\left\{\langle [\gamma]_m\rangle:\gamma\in \K\right\}$ is  independent; that $\langle [\lambda]_m\rangle$ has the desired isomorphism type; and that for each $\gamma\in \K\setminus \{\lambda\}$ we have $\langle [\gamma]_m \rangle\simeq R_mG_{e(\gamma)}$ where
\begin{equation*}
    e(\gamma)=\min \{i\in\{-\infty, 0,1\} :
    \gamma\in K_i^\times\}.
\end{equation*}

We proceed by induction; the case $m=1$ follows from the construction of $\K$. Suppose that for some $m\geq 2$ the set $\left\{\langle [\gamma]_{m-1}\rangle:\gamma\in \K\right\}$ is independent, and that for each $\gamma\in \K\setminus\{\lambda\}$ we have $\langle [\gamma]_{m-1}\rangle \simeq R_{m-1}G_{e(\gamma)}$.  We show that the statements hold for $m$ as well.

Consider $\langle [\gamma]_m\rangle$ for $\gamma\in \K\setminus \{\lambda\}$, and --- for simplicity --- write $e=e(\gamma)$. Since $\gamma\in K_{e}^\times$, we know $\langle [\gamma]_m\rangle$ is an $R_mG_e$-module. By induction $\langle [\gamma]_{m-1}\rangle$ is a free $R_{m-1}G_e$-module, and by Lemma~\ref{le:idealrmgiII},
\begin{equation*}
    2^{m-2}P(e,0)[\gamma]_{m-1} \neq [0]_{m-1}.
\end{equation*}
Moreover, $\langle [\gamma]_{m-1}\rangle$ is independent from $\langle [\lambda]_{m-1}\rangle$. Now suppose $\langle [\gamma]_m\rangle$ is not free as an $R_mG_e$-module. Then by Lemma~\ref{le:idealrmgiII},
\begin{equation*}
    2^{m-1}P(e,0)[\gamma]_m = [0]_m.
\end{equation*}
Taking square roots yields
%\begin{equation*}
%    (N_{K_e/F}\gamma)^{2^{m-1}} = \beta^{2^m}, \quad \beta\in K^\times.
%\end{equation*}
%We take square roots, yielding
%\begin{equation*}
%    (N_{K_e/F}\gamma)^{2^{m-2}} = (-1)^c + \beta^{2^{m-1}}
%\end{equation*}
%for some $c\in \Z$. Therefore
\begin{equation*}
    2^{m-2}P(e,0)[\gamma]_{m-1} = c[-1]_{m-1}.
\end{equation*}
for some $c \in \Z$. Now if $c=0$ or $m\le \nu$ then $c[-1]_{m-1}=[0]_{m-1}$ and so $2^{m-2}P(e,0)[\gamma]_{m-1} = [0]_{m-1}$, a contradiction.
Otherwise, $c[-1]_{m-1}\neq [0]_{m-1}$ and yet $c[-1]_{m-1}\in \langle [\lambda]_{m-1}\rangle$.  Then
\begin{equation*}
    [0]_{m-1}\neq c[-1]_{m-1} \in \langle [\gamma]_{m-1} \rangle \cap \langle
    [\lambda]_{m-1}\rangle,
\end{equation*}
a contradiction.  Hence $\langle [\gamma]_m\rangle$ is a free $R_m G_e$-module for all $\gamma\in \K\setminus \{\lambda\}$.

Now we claim that $\left\{\langle [\gamma]_m \rangle: \gamma\in \K\setminus \{\lambda\}\right\}$ is independent. By induction the set $\left\{\langle [\gamma]_{m-1}\rangle:\gamma\in \K\right\}$ is independent. We proceed by induction on the number of modules $\langle [\gamma]_m \rangle$ that we already know to be independent. The base case is clear. Now consider $\langle [\gamma]_m \rangle$ and $\oplus \langle [\gamma_i]_m \rangle$. By Lemma~\ref{le:exclII} it suffices to show that $\left(\langle[\gamma]_m\rangle\right)^\star \cap \oplus \left(\langle[\gamma]_m\rangle\right)^\star = \{0\}$.

Suppose the contrary.  Since each $\langle [\gamma_i]_m\rangle$ is a free $R_mG_{e(\gamma_i)}$-module,  we see that Lemma~\ref{le:starrmgiII} gives
\begin{equation*}
    \langle[\gamma_i]_m\rangle^\star = \langle 2^{m-1}P(e(\gamma_i),0)\rangle,
\end{equation*}
and this module is isomorphic to $\Ft$.  Hence we must have a
relation
\begin{equation*}
    c_\gamma 2^{m-1}P(e,0)[\gamma]_m + \sum c_{\gamma_i}
    2^{m-1}P(e(\gamma_i),0)[\gamma_i]_m = [0]_m,
\end{equation*}
with $c_\gamma$ and $c_{\gamma_i}\in \Z$ and not all divisible by
$2$.  Taking square roots, we obtain
\begin{equation*}
    c_\gamma 2^{m-2}P(e,0)[\gamma]_{m-1} + \sum c_{\gamma_i} 2^{m-2}P(e(\gamma_i),0)[\gamma_i]_{m-1} = c[-1]_{m-1}.
\end{equation*}
If $c[-1]_{m-1}= [0]_{m-1}$ then we have a contradiction of the
independence of  $\{\langle [\gamma]_{m-1}\rangle:\gamma \in \K\setminus\{\lambda\}\}$.  If
$c[-1]_{m-1}\neq [0]_{m-1}$, we have a contradiction of the
independence of $\langle [\lambda]_{m-1}\rangle$ from
$\sum_{\gamma\in \K\setminus \{\lambda\}} \langle [\gamma]_{m-1}\rangle$.

Similarly we show that $\langle [\lambda]_m \rangle$ is
independent from $\bigoplus_{\gamma\in \K\setminus \{\lambda\}} \langle [\gamma]_m\rangle$.
Suppose not:
\begin{equation*}
    r_\lambda[\lambda]_m = \sum r_i[\gamma_i]_m \neq [0]_m, \quad r_\lambda,r_i\in R_mG
\end{equation*}
for some finite set $\{\gamma_i\}\subset \K\setminus \{\lambda\}$.  Observe that, considered in $J_1$, this relation must be trivial.  Since $[\lambda]_1$ is a free $\Ft G$-module, this means that $r_\lambda \in \langle 2 \rangle \subset R_mG$; in the same way, $r_i \in \langle 2 \rangle \subset R_mG$ for any $\gamma_i \in \K$ with $\gamma \not\in K_0^\times$.  For those $\gamma_i \in \K \cap K_0^{\times}$, on the other hand, $\sigma$ acts trivially, and so $r_i$ can be assumed to be in $R_m$.  From this we again conclude $r_i \in \langle 2 \rangle \subset R_mG$ in this case as well.  

Write $r_i=2s_i$ for all $i$ and $r_\lambda=2s_\lambda$
for $s_i, s_\lambda\in R_mG$.  Then we take square roots:
\begin{equation*}
    s_\lambda[\lambda]_{m-1} + c[-1]_{m-1} = \sum s_i[\gamma_i]_{m-1}, \quad
    c\in \Z.
\end{equation*} modules $\langle [\lambda]_{m-1}\rangle$ and $\sum \langle [\gamma_i]_{m-1}\rangle$ are independent by induction, we must have that both sides of the equation are trivial.  Squaring the right-hand side, we reach $\sum r_i[\gamma_i]_m = [0]_m$, contrary to hypothesis.  Hence all $\langle [\gamma]_m\rangle$ for $\gamma\in \K$ are independent.

Now we consider the isomorphism class of $\langle [\lambda]_m\rangle$, and we will show that $\ann_{R_mG}\langle [\lambda]_m \rangle = \langle 2^\nu(\sigma-1)\rangle$.    Since
\begin{equation*}
    (\lambda^{(\sigma-1)})^{2^{\nu}} = (\xi_{2^\nu}^{-1})^{2^\nu}= 1,
\end{equation*}
 $2^{\nu}(\sigma-1)\in \ann \langle [\lambda]_m\rangle$ for all $m$. Now we show that
\begin{equation*}
\ann_{R_mG}
\langle [\lambda]_m\rangle \subset \langle 2^{\nu}(\sigma-1)\rangle.
\end{equation*}
Suppose $r[\lambda]_m = 0$. Then $\lambda^r=\beta^{2^m}$ for
$\beta\in K^\times$. Since $\langle[\lambda]_1\rangle$ is a free
$\Ft G$-module from above, we obtain $r\in \langle 2\rangle\in R_m
G$. Write $r=2s$ for $s\in R_mG$.  Then we may take square roots:
$\lambda^s = (-1)^c \beta^{2^{m-1}}$ for $c\in \Z$.  Since
$-1=(\lambda^{\sigma-1})^{2^{\nu-1}}$, we obtain
\begin{equation*}
    (s-c2^{\nu-1}(\sigma-1))[\lambda]_{m-1} = 0.
\end{equation*}
By induction $s-c2^{\nu-1}(\sigma-1) = k2^\nu(\sigma-1)$ for some $k \in R_{m-1}G$. Then
\begin{equation*}
r = 2s=2(c2^{\nu-1}(\sigma-1)+k2^\nu(\sigma-1)) \in \langle 2^\nu(\sigma-1)\rangle \subset R_mG.\end{equation*}
Hence $\ann_{R_mG} \langle \lambda \rangle = \langle
2^{\nu}(\sigma-1)\rangle$, as desired. 

Finally, $\langle [\lambda]_m \rangle$ is indecomposable as an $R_mG$-module because $R_mG$ is a local ring and all cyclic modules over a local ring are indecomposable.

Now let $Y_0$ be the span of $\langle [\gamma]_m\rangle$ for $\gamma\in \K \cap K_0^\times$ and $Y_1$ be the span of $\langle [\gamma]_m\rangle$ for $\gamma\in \K \setminus (K_0^\times \cup \{\lambda\})$. We only have left to determine the ranks of these modules.  If we write $e_1(\K) = |\K\setminus K_0^\times|$ and $e_0(\K) = |\K \cap K_0^\times|$, then \cite[Cor.~1]{MSS} gives $e_1(\K) = \dim_{\F_2}\left([N_{K/F}(K^\times)]_1\right)$ and $|\K| = \dim_{\F_2}\left([F^\times]_1\right)$.  Because $\xi_2 \in F$ by assumption, the natural map $F/F^{\times 2} \to J_1$ has 2 elements in its kernel: namely the classes represented by $1$ and $a$, where $K = F(\root\of{a})$.  Because $N_{K/F}(\root\of{a}) = -a$ and $-1\not\in N_{K/F}(K^\times)$, we cannot have $a = N_{K/F}(\beta)f^2$ for $\beta \in K^\times$ and $f \in F^\times$.  Hence $\dim_{\F_2}\left([N_{K/F}(K^\times)]_1\right) = \dim_{\F_2}\left(N_{K/F}(K^\times)F^{\times 2}/F^{\times 2}\right)$, and so it follows that $$e_1(\K) = \dim_{\F_2}\left([N_{K/F}(K^\times)]_1\right) = e_1(K/F).$$  On the other hand $aF^{\times 2} \in F^\times/F^{\times 2}$, and so $$e_0(\K) = \codim_{\F_2}\left([N_{K/F}(K^\times)]_1 \subset [F^\times]_1\right) = e_0(K/F)-1.$$  By construction we have that $Y_0$ is rank $e_0(\K) = e_0(K/F)-1$, and $Y_1$ is rank $e_1(\K)-1 = e_1(K/F)-1$.

The inverse limit calculation follows naturally.
\end{proof}

\section{Exceptionality}\label{se:np1}

For the rest of the paper we focus on the case where $F$ contains a primitive $p$th root of unity $\xi_p$, and when $p=2$ and $n=1$ we also assume $-1 \in N_{K/F}(K^\times)$.  To make equations easier to read, we will also begin to write $K^\times$ as an additive group; in particular we will use $0$ to indicate $1 \in K^\times$, and we'll write $\frac{1}{p}\gamma$ to indicate $\root{p}\of{\gamma}$ for $\gamma \in K^{\times p}$.  For the remainder of the paper, we also adopt additive notation for $K^\times$, and in particular write the $\F_pG$-action multiplicatively.  This helps make equations easier to read.

To motivate the results we will develop in this section (and, indeed, subsequent sections as well), recall that a decomposition for $J_m$ requires us to find generators of $J_1$ whose relations in $J_m$ can be computed.  As we saw (both in the introduction and the proofs of Theorems \ref{th:p2n1m1notnorm} and \ref{th:noxip}), controlling the appearance of a primitive $p$th root of unity is critical to teasing out ``new" relations which appear when we move from $J_{m-1}$ to $J_m$.

Under the hypotheses of this section ($\xi_p \in F$, and $-1\in N_{K/F}(K^\times)$ if $p=2$ and $n=1$), it was shown in \cite{MSS} that $J_1$ is a direct sum of cyclic modules with at most one summand non-free (over some quotient $\F_pG_i$).  The generator for this ``exceptional summand" is chosen in the following way: one defines an invariant $i(K/F)$ as
\begin{align*}
    i(K/F) &= \min \{\ i\in \{-\infty, 0, 1, \dots, n\} :  \exists
    \gamma\in K^\times \text{ such that } \\ &\qquad\qquad
    N_{K/F}(\gamma) \not\in pF^{\times} \text{\ and\ } \\
    &\qquad\qquad (\tau-1)[\gamma]_1 \in [K_i^\times]_1\ \ \ \forall
    \tau\in\Gal(K/F)\},
\end{align*}
and then a generator for the exceptional summand can be chosen as any $\gamma\in K^\times$ representing the minimum in the definition above.  We call such an element $\gamma$ MSS-exceptional.

Hence in computing a decomposition for $J_m$, one of our key considerations is to determine which MSS-exceptional element will allow us to keep track of relations in $J_m$.  A full answer to this question requires a great deal of machinery (this is the overall goal of sections \ref{se:np1}--\ref{sec:exceptional.modules}), but begins with the introduction of norm pairs in this section.

%Recall that $K/F$ is a cyclic field extension of degree $p^n$ with group $G$ and fixed generator $\sigma$, and that $G_i=\Gal(K_i/F)$, where $K_i$ is the subextension of degree $p^i$ over $F$. 

%In this section we introduce the notion of norm pair, show that norm pairs exist for all extensions $K/F$ as described in the previous paragraph, and define a partial order on the set of norm pairs.  In doing so we will be most interested in minimal norm pairs, for they capture information about $K/F$ which will be essential in the decomposition of $J_m$.  
%Ultimately, after we use a minimal norm pair to produce a decomposition of $J_m$ into a direct sum of indecomposables, we will be able to express the data in the norm pair in simpler terms.

%The norm pair describes the existence of an element $\alpha$ and a set of elements $\{\delta_i\}$ such that, first, $\alpha^{\sigma-d}$ may be expressed as a product of $p$-th powers of the elements of $\{\delta_i\}$, for some $d\in \Z$, and, second, a primitive $p$th root of unity bears a special relationship to $\alpha$ and $\{\delta_i\}$.  The first condition may be considered a basic equation in $J_m$, but the second condition describes a crucial connection with the structure of the multiplicative group $K^\times$ of $K$.  We keep track not only of $d$ but also of the subfields $K_{a_i}$ of $K/F$ which contain, respectively, the elements $\delta_i$, for each $\delta_i$, $0\le i<m$. In this way a norm pair consists of an integer $d$ and a set of $m$ values $a_{i}$, $0\le i<m$, such that $\delta_i\in K_{a_i}^\times$.

\subsection{Norm pairs and their basic properties}\label{sec:norm.pair.subsection}

\begin{definition*}[Norm Pair]
    Let $m\in \N$, and suppose that $\mathbf{a}$ is a vector $\mathbf{a} = (a_0,\dots,a_{m-1}) \in \{-\infty,0,1,\dots,n\}^m$ with $a_0<n$, and that $d \in U_1$.  We say that $(\mathbf{a},d)$ is a \emph{norm pair of length $m$} if there
    exist $\alpha,\delta_m \in K^\times$ and $\delta_i\in K_{a_i}^\times$ for each $0 \leq i \leq m-1$ satisfying
    \begin{equation}\label{eq:defining.properties.for.exceptionality}
    \begin{split}
        {(\sigma-d)}\alpha &= \sum_{i=0}^m p^i \delta_i\\
        \xi_p &= \left(\frac{d-1}{p}\right)N_{K/F}(\alpha)+
        N_{K_{n-1}/F}(\delta_0) + \left(\sum_{i=1}^{m}p^{i-1}
        N_{K/F}(\delta_i) \right).
    \end{split}
    \end{equation}
    We say that $(\alpha,\{\delta_i\}_{i=0}^m)$ \emph{represents}
    $(\mathbf{a},d)$. We call $\mathbf{a}$ a \emph{norm vector of length
    $m$} and $d$ a \emph{twist}.
\end{definition*}

%We will see in Proposition~\ref{pr:norm.pair.independent.of.xi.p} that the definition of norm pair is independent of the choice of primitive $p$th root of unity $\xi_p\in F$.  Moreover, we show that if $(\mathbf{a},d)$ is a norm pair of length $m$, then so is $(\mathbf{a},\check d)$ for any $\check d = d\ (p^m)$; the set of twists $d$ associated to $\mathbf{a}$ is a union of cosets of $U_m$ in $U_1$.

We impose a partial ordering $\preceq_m$ on norm pairs of length $m\in \N$, as follows.  We order the vectors $\mathbf{a}$ lexicographically.  For twists $d, d'\in U_1$, we say that $d\le_m d'$ if $d' \in U_i$ implies $d \in U_i$ for any $1 \leq i \leq m$ --- except in the case that $p=2$ and $d, d' \notin U_2$.  In this case $d,d' \in -U_2 := \{z \in \Z: z \equiv -1\ (4)\}$, and we say that $d \leq_m d'$ if $d' \in -U_i$ implies $d \in -U_i$ for all $2 \leq i \leq m$.  With these partial orders on vectors and twists, we order norm pairs $(\mathbf{a},d)$ lexicographically in the pair. Hence $(\mathbf{a},d)\preceq_m (\mathbf{a}',d')$ if either $\mathbf{a} < \mathbf{a}'$, or if $\mathbf{a} = \mathbf{a}'$ and $d\le_m d'$.

\begin{definition*}[Minimal Norm Pair, $m$-exceptional element]
  Let $m\in \N$.  If $(\mathbf{a},d)$ is a norm pair of length $m$ so that for any other norm pair $(\mathbf{a}',d')$ of length $m$ we have $(\mathbf{a},d) \preceq_m (\mathbf{a}',d')$, then we say that $(\mathbf{a},d)$ is a \emph{minimal norm pair of length $m$}.  If $(\alpha,\{\delta_i\}_{i=0}^m)$ represents a minimal norm pair $(\mathbf{a},d)$, then we say that $\alpha$ is an \emph{$m$-exceptional element}.
\end{definition*}

\begin{remark*}
The curious reader might wonder about the relationship between this new notion of exceptionality and the notion of MSS-exceptionality discussed above.  This will be considered in Proposition \ref{pr:MSS} below.
\end{remark*}

\begin{example*}
Consider the extension $K/F$ given by $\mathbb{Q}(\xi_{p^{n+1}})/\mathbb{Q}(\xi_p)$, where here $p$ is an odd prime and $\xi_{p^{n+1}}$ is some primitive $p^{n+1}$st root of unity in $\mathbb{C}$; to make things simple, we'll assume that $\xi_{p^{n+1}}^{p^n} = \xi_p$ (or, in the additive notation we use throughout the rest of the paper, that $p^n\xi_{p^{n+1}} = \xi_p$).  This extension has $\Gal(K/F) = \langle \sigma \rangle \simeq \mathbb{Z}/p^n\mathbb{Z}$.  We claim that if we let $\mathbf{a} = \{-\infty,\cdots,-\infty\}$, and if we choose $d$ so that $\sigma(\xi_{p^{n+1}}) = d \xi_{p^{n+1}}$, then $(\mathbf{a},d)$ is a minimal norm pair for this extension.  

First, let us check that $d \in U_1$, which means we simply need to argue that $d-1 = pk$ for some $k \in \mathbb{Z}$.  If this were not the case, then we'd have $d-1=j+pk$ for some $j \in \{1,2,\cdots,p-1\}$.  On the one hand, we'd have $(\sigma-1)\left(p^n \xi_{p^{n+1}}\right) = (\sigma-1)\xi_p = 0$.  But on the other hand we'd have $$(\sigma-1)(p^n \xi_{p^{n+1}}) = p^n \left( (\sigma-1) \xi_{p^{n+1}}\right) = p^n \left((d-1)\xi_{p^{n+1}}\right) = p^n \left((j+pk)\xi_{p^{n+1}}\right) = j \xi_p.$$  This is clearly a contradiction, and so $d-1=pk$ for some $k$.  In fact, a similar argument will tell us that $d \not\in U_2$: if instead we had $d-1=p^2k$ for some $k \in \mathbb{Z}$, then on the one hand we'd get $(\sigma-1)\left(p^{n-1} \xi_{p^{n+1}}\right) = (\sigma-1)\xi_{p^2} \neq 0,$ but on the other hand we'd have $$(\sigma-1)\left(p^{n-1} \xi_{p^{n+1}}\right) = p^{n-1}\left((\sigma-1) \xi_{p^{n+1}}\right) = p^{n-1}\left((d-1)\xi_{p^{n+1}}\right) = p^{n-1}\left(p^2k \xi_{p^{n+1}}\right) =0.$$

Before we can verify that $(\mathbf{a},d)$ is a norm pair, we need an additional preliminary result: that $N_{K/F}(\xi_{p^{n+1}}) = \xi_p^c$ for some $c \in \{1,2,\cdots,p-1\}$.  One can show this inductively, but we'll focus on the base case (when $n=1$) here.  In this case we have $$(N_{K/F}(\xi_{p^2}))^p = N_{K/F}(\xi_{p^2}^p) = N_{K/F}(\xi_p) = 1,$$ so that $N_{K/F}(\xi_{p^2})$ is some $p$th root of unity.  To argue that this is a primitive $p$th root of unity, note that otherwise we'd have $\xi_{p^2} \in \ker N_{K/F}$, and so by Hilbert's Theorem 90 we'd have $\xi_{p^2} = (1-\sigma)\beta$ for some $\beta \in \mathbb{Q}(\xi_{p^2})$.  But this would force $[\mathbb{Q}(\xi_p,\beta):\mathbb{Q}(\xi_p)]=p^2$, despite the fact that $\mathbb{Q}(\xi_p,\beta)$ is an intermediate extension within the degree $p$ extension $\mathbb{Q}(\xi_{p^2})/\mathbb{Q}(\xi_p)$.  

We are now ready to proceed.  Since $d \in U_1 \setminus U_2$, we know that $d-1 = jp + kp^2$ for some $j \in \{1,2,\cdots,p-1\}$.  Recalling that $N_{K/F}(\xi_{p^{n+1}}) = \xi_p^c$ for some $c$, choose $t \in \mathbb{Z}$ so that $jtc \equiv 1 \pmod{p}$. Define $\alpha = \xi_{p^{n+1}}^t$, and for each $0 \leq i \leq m$ let $\delta_i=0$.  Then by our choice of $d$ we know that $(\sigma-d)\alpha = 0 = \sum_{i=0}^m p^i\delta_i$.  Furthermore we have $$\left(\frac{d-1}{p}\right)N_{K/F}(\alpha)+N_{K_{n-1}/F}(\delta_0) + \left(\sum_{i=1}^m p^{i-1} N_{K/F}(\delta_i)\right) = (j+pk)t c \xi_p = \xi_p.$$  Hence condition (\ref{eq:defining.properties.for.exceptionality}) is satisfied, and $(\mathbf{a},d)$ is a norm pair.  

To see that it is a minimal norm pair, note that any other norm pair $(\mathbf{a}',d')$ with $\mathbf{a}' \neq (-\infty,\cdots,-\infty)$ has $(\mathbf{a},d) \preceq_m (\mathbf{a}',d')$ since $\mathbf{a}<\mathbf{a}'$.  If we know that $\mathbf{a}'=\mathbf{a} = (-\infty,\cdots,-\infty)$, it is a bit harder to argue that $d \leq_m d'$.  The result is confirmed by \cite[Theorem 1]{MSS2c}.
\end{example*}

We now establish some basic properties of norm pairs of length $m$: existence, well-definedness, and their behavior when varying $m$.  

\begin{proposition}\label{pr:norm.pairs.exist}
    For every $m\in \N$, norm pairs of length $m$ exist, and hence minimal norm pairs of length $m$ exist.
\end{proposition}

\begin{proof}%[Proof of Proposition~\ref{pr:norm.pairs.exist}]
Albert's result \cite[Theorem~3]{Al} gives $\xi_p=N_{K_{n-1}/F}(\delta_0)$ for $\delta_0\in K_{n-1}^\times$. Then $N_{K/F}(\delta_0)=0$ and by Hilbert 90 there exists $\alpha\in K^\times$ such that ${(\sigma-1)}\alpha = \delta_0$. Let $m\in \N$ be arbitrary, and set $d=1$ and $\delta_i=0$ for $1\le i\le m$.  Then we have
    \begin{equation*}
        \sigma\alpha = d\alpha+\delta_0 +p\delta_1+\cdots
        +p^m\delta_m
    \end{equation*}
    and
    \begin{equation*}
        \xi_p = N_{K_{n-1}/F}(\delta_0).
    \end{equation*}
    Hence
    \begin{equation*}
        (\mathbf{a},d):=((n-1,\underbrace{-\infty,\dots, -\infty}_{m-1}),1)
    \end{equation*}
    is a norm pair of length $m$ with twist $1$.  Therefore norm pairs exist.
\end{proof}

%\noindent (Recall that for this section we assume that we are not in the case $p=2$, $n=1$, and $-1\not\in N_{K/F}(K^\times)$.)

\begin{proposition}\label{pr:norm.pair.independent.of.xi.p}
 Norm pairs are independent of the choice of primitive $p$th root of unity $\xi_p\in F$. For $m\in \N$, if $(\mathbf{a},d)$ is a norm pair of length $m$ and $\check d=d\ (p^m)$, then $(\mathbf{a},\check d)$ is also a norm pair of length $m$.
\end{proposition}

\begin{proof}%[Proof of Proposition~\ref{pr:norm.pair.independent.of.xi.p}]
If $(\alpha,\{\delta_i\})$ represents a norm pair $(\mathbf{a},d)$, then for any $t$ relatively prime to $p$, we have that  $(t\alpha,\{t\delta_i\})$ satisfies condition (\ref{eq:defining.properties.for.exceptionality}) with $\xi_p$ replaced by $t\xi_p$.  Hence norm pairs are independent of the choice of primitive $p$th root of unity $\xi_p$ in $F$.

    Now suppose that $(\mathbf{a},d)$ is a norm pair of length $m$, $(\alpha,\{\delta_i\})$ represents $(\mathbf{a},d)$, and let $\check d=d \bmod p^m$.  Then $\check d=d+p^mx$ for some $x\in \Z$. We set $\check\delta_i = \delta_i$ for each $0\le i<m$ and define $\check\delta_m =\delta_m-x \alpha$.  Then we calculate
    \begin{align*}
    	(\sigma-\check d)\alpha &= (\sigma-d)\alpha + (d-\check d)\alpha
		= \sum_{i=0}^{m}p^i \delta_i -p^m x \alpha
        = \sum_{i=0}^m p^i \check\delta_i.
    \end{align*}
    and similarly
    \begin{align*}
        \xi_p %&= \xi_p \cdot N_{K/F}(\alpha)^{p^{m-1}x}\cdot N_{K/F}(\alpha)^{-p^{m-1}x} \\
        %&=\xi_p \cdot N_{K/F}(\alpha)^{\frac{\check d-d}{p}} \cdot N_{K/F}(\alpha)^{-p^{m-1}x}\\
        &=\left(\frac{d-1}{p}\right)N_{K/F}(\alpha)+ N_{K_{n-1}/F}(\delta_0) + \left(\sum_{i=1}^{m} p^{i-1}N_{K/F}(\delta_i)\right) \\&\quad \quad \quad + \left(\frac{\check d-d}{p}-xp^{m-1}\right) N_{K/F}(\alpha) \\
%        &=N_{K/F}(\alpha)^{\frac{d-1}{p}+\frac{\check d-d}{p}}\cdot N_{K_{n-1}/F}(\check \delta_0) \cdot \left(\prod_{i=1}^{m} N_{K/F}(\check \delta_i)^{p^{i-1}} \right)  \\
        &=\left(\frac{\check d-1}{p}\right)N_{K/F}(\alpha)+ \left(\sum_{i=1}^{m}p^{i-1} N_{K/F}(\check\delta_i)\right)
    \end{align*}
    Hence $(\alpha,\{\check\delta\})$ represents $(\mathbf{a},\check d)$,
    and so $(\mathbf{a},\check d)$ is a norm pair as well.
\end{proof}

If we are given two norm pairs $(\mathbf{a},d) = ((a_0,\dots,a_{m-1}),d)$ and $(\mathbf{a}',d') = ((a_0',\dots, a_{s-1}'),d')$ of respective lengths $m<s$, we say that $(\mathbf{a}',d')$ \emph{extends} $(\mathbf{a},d)$ if $(a_0,\dots,a_{m-1}) = (a_0',\dots,a_{m-1}')$ and $d \le_m d' \le_m d$.

\begin{proposition}\label{pr:norm.pair.extension.and.contraction}
 Suppose $1\le m<s$.  Then every minimal norm pair of length $m$ extends to a minimal norm pair of length $s$, and every minimal norm pair of length $s$ is the extension of a minimal norm pair of length $m$.
\end{proposition}

\begin{proof}%[Proof of Propositions~\ref{pr:norm.pair.extension.and.contraction} and \ref{pr:contacting.exceptional.elements}]
%    We consider Proposition~\ref{pr:norm.pair.extension.and.contraction} first. 
Let
    \begin{align*}
      (\mathbf{a},d) &= ((a_0,\dots,a_{m-1}),d)\\
      (\mathbf{a}', d') &= ((a_0',\dots,a_{s-1}'),d')
    \end{align*}
    be minimal norm pairs of length $m<s$, respectively.
    Suppose that $(\alpha,\{\delta_i\})$ represents $(\mathbf{a},d)$
    and $(\check\alpha,\{\check\delta_i\})$ represents $(\mathbf{a}',d')$.

    Because
    \begin{align*}
        (\sigma-d)\alpha &= \sum_{i=0}^m p^i \delta_i\\
%    \end{equation*}
%    and
%    \begin{equation*}
        \xi_p &= \left(\frac{d-1}{p}\right)N_{K/F}(\alpha)+
        N_{K_{n-1}/F}(\delta_0) + \sum_{i=1}^{m}
        p^{i-1}N_{K/F}(\delta_i),
    \end{align*}
    we see that by setting $\delta_i=0$ for each $m<i\le s$, the norm
    pair $(\mathbf{a},d)=((a_0,\dots,a_{m-1}),d)$ may be extended to a
    norm pair of length $s>m$
    \begin{equation*}
        ((a_0,\dots,a_{m-1},n,\underbrace{-\infty, \dots,-\infty}_{s-m-1}),d).
    \end{equation*}
    (Certainly $\delta_m\in K_n^\times=K^\times$, though it may also
    be in some proper subfield as well.)

    Similarly, since
    \begin{align*}
        (\sigma-\check d)\check \alpha &= 
        \left(\sum_{i=0}^{m-1} p^i{\check\delta}_i\right) +
        p^m\left(\sum_{i=m}^{\check m} p^{i-m}{\check\delta}_i\right)\\
%    \end{equation*}
%    and
%    \begin{align*}
        \xi_p &= \left(\frac{\check d-1}{p}\right)N_{K/F}(\check\alpha)+
        N_{K_{n-1}/F}(\check \delta_0) + \sum_{i=1}^{m-1}
        p^{i-1}N_{K/F}(\check \delta_i)\\&\quad \quad \quad \quad +
        p^{m-1}N_{K/F}\left(\sum_{i=m}^{\check m} p^{i-m}{\check\delta}_i\right)
    \end{align*}
    we may obtain from $(\mathbf{a}',d')$ a finite norm pair
    $((a_0',\dots,a_{m-1}'),d')$ of length
    $m<s$.

    Now by minimality for each length $m, s$ we have
    \begin{align*}
        (\mathbf{a},d) &= ((a_0,\dots,a_{m-1}),d) \preceq_m ((
        a_0',\dots,
        a_{m-1}'),d') \\
        (\mathbf{a}',d') &= ((a_0',\dots,a_{s-1}'), d') \preceq_{s} ((a_0,\dots,a_{m-1}, n,
        -\infty,\dots,-\infty),d).
    \end{align*}
    We may truncate the vectors in the second equation to have length $m$ while still preserving their order (though now under the $\preceq_m$ order instead of $\preceq_{s}$).  This leaves
    \begin{equation*}
        ((a_0,\dots,a_{m-1}),d) \preceq_m ((a_0',\dots,
        a_{m-1}'), d') \preceq_m
        ((a_0,\dots,a_{m-1}),d).
    \end{equation*}
    Hence $a_i=a_i'$ for all $0\le i< m$ and
    \begin{equation*}
      d\le_m \check d\le_m d.
    \end{equation*}
%    Since by Proposition~\ref{pr:norm.pairs.exist}, minimal norm pairs exist of any length, we see that minimal norm pairs of length $m$ may be extended to minimal norm pairs of arbitrary length $s>m$, and minimal norm pairs of length $s$ are extensions of minimal norm pairs of lengths $1\le m<s$.  We have therefore established Proposition~\ref{pr:norm.pair.extension.and.contraction}.
%    For Proposition~\ref{pr:contacting.exceptional.elements}, observe that we have already shown that an $s$-exceptional element $\alpha$ is $m$-exceptional for all $1\le m<s$, and that the associated $\delta_i$ may be chosen as stated in the Proposition.
\end{proof}

\begin{corollary}\label{cor:contacting.exceptional.elements}
If $\alpha$ is an $s$-exceptional element, then $\alpha$ is $m$-exceptional for all $1\le m<s$.  More precisely, if $(\alpha,\{\check \delta_i\}_{i=0}^{s})$ represents a minimal norm pair $(\mathbf{a},d)$ of length $s$, then $(\alpha,\{ \delta_i\}_{i=0}^{m})$ represents the minimal norm pair $((a_0,\dots,a_{m-1}),d)$ of length $m$, where $\delta_i = \check\delta_i$ for all $0\le i<m$, and $\delta_{m} = \sum_{i=m}^{\check m}  p^{i-m}\check \delta_i$.
\end{corollary}

\subsection{Connection to an earlier notion of exceptionality}

Recall that we defined MSS-exceptional elements at the beginning of this section.  We now connect this notion of exceptionality and the notion of $1$-exceptionality.

We start with some preparatory results.

\begin{lemma}\label{le:equivalent.ways.to.express.power.of.xi.as.product.of.norms}
Assume that $\xi_p\in F^\times$ and if $p=2$ then $n>1$.  The following are equivalent for $\gamma\in K_{n-1}^\times$,
    $\delta\in K^\times$, and $c\in \Z$:
    \begin{enumerate}
        \item\label{it:l14a} there exists $\alpha\in K^\times$ such
        that
        \begin{enumerate}
            \item\label{it:l14sa} $(\sigma-1)\alpha=\gamma +
            p\delta$
            \item\label{it:l14sb} $\frp N_{K/F}(\alpha)\in K^\times$
            and
            \item\label{it:l14sc} $c \xi_p=(\sigma-1) \frp
            {N_{K/F}(\alpha)}$
        \end{enumerate}
        \item\label{it:l14b} $c \xi_p=N_{K_{n-1}/F}(\gamma) +
        N_{K/F}(\delta)$
    \end{enumerate}
\end{lemma}

\begin{proof}
    (\ref{it:l14a})$\Longrightarrow$(\ref{it:l14b}) is \cite[Lemma 14]{MSS}.

    For (\ref{it:l14b})$\Longrightarrow$(\ref{it:l14a}), take $p$th
    powers of the equation to obtain
    \begin{equation*}
        0 = pN_{K_{n-1}/F}(\gamma) + pN_{K/F}(\delta)
        =   N_{K/F}(\gamma+ p\delta).
    \end{equation*}
    Applying Hilbert 90, we obtain $\alpha\in K^\times$ such that
    $(\sigma-1)\alpha = \gamma+p\delta$.  Hence we have
    part~(\ref{it:l14sa}).

    Since $P(i,0)=(\sigma-1)^{p^i-1}$ in $\Fp G$, then modulo $pK^\times$ we have
    \begin{align*}
        [N_{K/F}(\alpha)]_1 &= P(n,0)[\alpha]_1 = (\sigma-1)^{p^n-1}[\alpha]_1 \\
        &= (\sigma-1)^{p^{n}-2}([\gamma]_1+p[\delta]_1) \\ 
        &= (\sigma-1)^{p^{n}-2}([\gamma]_1) \\ 
%        &= (\sigma-1)^{p^{n}-p^{n-1}-1}(\sigma-1)^{p^{n-1}-1}(\gamma)\\
        &= (\sigma-1)^{p^n-p^{n-1}-1}P(n-1,0)[\gamma]_1 \\
        &= (\sigma-1)^{p^n-p^{n-1}-1}[N_{K_{n-1}/F}(\gamma)]_1 \\
        &= [0]_1.
    \end{align*}
    For the last equation, we use the fact that if $p=2$, $n>1$ and so $p^n-p^{n-1}-1\ge 1$.   Hence we have part~(\ref{it:l14sb}).

    Since $N_{K/F}(\alpha) \in pK^\times \cap F^\times$, for some $\tilde c\in \Z$ we have
    \begin{equation*}
        (\sigma-1) \frp N_{K/F}(\alpha)=\tilde c \xi_p.
    \end{equation*}
    From (\ref{it:l14a})$\Longrightarrow$(\ref{it:l14b}) we obtain
    \begin{equation*}
        \tilde c \xi_p=N_{K_{n-1}/F}(\gamma) + N_{K/F}(\delta) = c \xi_p.
    \end{equation*}
    Therefore $\tilde c \xi_p=c \xi_p$ and, in particular, $(\sigma-1) \frp  N_{K/F}(\alpha) =c \xi_p$.  Hence we have part~(\ref{it:l14sc}).
\end{proof}

We will use Lemma \ref{le:equivalent.ways.to.express.power.of.xi.as.product.of.norms} in our proof of Proposition \ref{pr:excprop} below, but first we need the following result to handle an issue when $p=2$, $n=1$.

\begin{lemma}\label{lem:p.2.and.n.1.gives.a0.minus.infinity} 
If $p=2$ and $n=1$ then $a_0=-\infty$. 
\end{lemma}
\begin{proof}
Suppose that $K = F(\frac{1}{2}a)$, and observe that $N_{K/F}(\frac{1}{2}a) = -1+a$.  By our overriding hypothesis, when $p=2$ and $n=1$ we must also have an element $\gamma \in K^\times$ satisfying $N_{K/F}(\gamma) = -1$.  Define $\alpha = \gamma + \frac{1}{2}a$, $\delta_0 = 0$ and $\delta_1 =-1\cdot \gamma$; we will show that $(\alpha,\{\delta_0,\delta_1\})$ represents the norm pair $((-\infty),1)$ of length $1$; by minimality of the norm pair $((a_0),d)$, it will follow that $a_0 = -\infty$.

To see that $(\alpha,\{\delta_0,\delta_1\})$ represents $((-\infty),1)$, we simply need to verify conditions (\ref{eq:defining.properties.for.exceptionality}).  To wit:
\begin{align*}
(\sigma-1)\alpha &= (\sigma-1)\gamma + (\sigma-1)\frac{1}{2}a\\
&= (\sigma+1)\gamma + (\sigma-1)\frac{1}{2}a - 2\cdot\gamma\\
&=-1 + -1 - 2\cdot\gamma\\
&= 2\delta_1\\[10pt]
-1 &= -1\cdot -1 = N_{K/F}(-1\cdot \gamma)\\
&= \frac{d-1}{2}N_{K/F}(\alpha) + N_{K_{n-1}/F}(\delta_0) + N_{K/F}(\delta_1).
\end{align*}
\end{proof}

\begin{proposition}\label{pr:excprop}
Let $\alpha\in K^\times$ be a $1$-exceptional element.  Then
    \begin{equation*}
        \frp N_{K/F}(\alpha)\in K^\times\quad \text{ and }
        \quad (\sigma-1)\frp N_{K/F}(\alpha) = \xi_p.
    \end{equation*}
    Moreover, $\alpha\not\in K_{n-1}^\times$.
\end{proposition}

\begin{proof}%[Proof of Proposition~\ref{pr:excprop}]
    Suppose that $\alpha\in K^\times$ is an $1$-exceptional element. Then $(\alpha,\{\delta_0,\delta_1\})$ represents a minimal norm pair $((a_0),d)$ of length $1$.  Hence $a_0<n$, $\delta_0\in K_{a_0}^\times$, and
    \begin{equation*}
        (\sigma-d)\alpha = \delta_0+p\delta_1, \quad\quad
        \xi_p=N_{K_{n-1}/F} (\delta_0) +
        N_{K/F}\left(\delta_1+\frac{d-1}{p}\alpha\right).
    \end{equation*}

    Now if we are not in the case $p=2$, $n=1$ then, letting $\gamma=\delta_0$ and $\delta = \delta_1 + \frac{d-1}{p}\alpha$, Lemma~\ref{le:equivalent.ways.to.express.power.of.xi.as.product.of.norms} gives the first two desired properties for $\alpha$.

    If $p=2$ and $n=1$ then by Lemma \ref{lem:p.2.and.n.1.gives.a0.minus.infinity},
    $a_0=-\infty$ and so $\delta_0=0$.
    Hence $(\sigma-d)\alpha = 2\delta_1$ and $-1=N_{K/F}(\delta)$, where again we write $\delta = \delta_1+\frac{d-1}{2}\alpha$.  We see that $(\sigma-1)\alpha=2
    \delta_1+(d-1)\alpha$ with $d-1\in 2\Z$, and then
    \begin{equation*}
        N_{K/F}(\alpha) = (\sigma+1)\alpha = (\sigma-1)\alpha +
        2\alpha \in 2K^\times.
    \end{equation*}
    The actions of $\sigma-1$ on the two square roots of
    $N_{K/F}(\alpha)$ are identical, and we calculate the action for
    one:
    \begin{align*}
        (\sigma-1)\frt N_{K/F}(\alpha) &= (\sigma-1)(\delta+\alpha)
        \\ &= (\sigma-1)(\delta) +2\delta\\ &= N_{K/F}(\delta) = -1.
    \end{align*}
    Hence $\alpha$ has the first two desired properties.

    For the last statement, observe that $N_{K/F}(K_{n-1}^\times)\subset pF^\times$. Since $\xi_p\in F^\times$, $\alpha\in K_{n-1}^\times$ would then imply $(\sigma-1)\frp N_{K/F}(\alpha)=0$, a contradiction.
\end{proof}

\begin{proposition}\label{pr:MSS}
Let $\alpha\in K^\times$ be a $1$-exceptional element.  Then
    \begin{enumerate}
        \item $\alpha$ is MSS-exceptional
        \item\label{it:MSS2} for $((a_0,\dots,a_{m-1}),d)$ a minimal norm pair, we have
        $a_0=i(K/F)$
        \item\label{it:MSS3} if $i(K/F)\neq -\infty$, $\langle
        (\sigma-1)[\alpha]_1\rangle$ is a free $\Fp
        G_{i(K/F)}$-module.
    \end{enumerate}
\end{proposition}

\begin{proof}%[Proof of Proposition~\ref{pr:MSS}]
    Since $\alpha$ is a $1$-exceptional element, there exist
    $\delta_0, \delta_1\in K^\times$ such that $(\alpha,
    \{\delta_0,\delta_1\})$ represents a minimal norm pair $(\mathbf{a},d)$.  Then $\delta_0\in K_{a_0}^\times$ and
    \begin{equation*}
        (\sigma-d)\alpha = \delta_0+p\delta_1, \quad\quad
        \xi_p=N_{K_{n-1}/F} (\delta_0) +
        N_{K/F}\left(\delta_1+\frac{d-1}{p}\alpha\right).
    \end{equation*}
    By Proposition~\ref{pr:excprop},
    \begin{equation*}
        (\sigma-1) \frp N_{K/F}(\alpha) \neq 0,
    \end{equation*}
    and so $N_{K/F}(\alpha)\not\in pF^{\times}$.

    Let
    \begin{equation*}
        r = \min \{i\in \{-\infty, 0,\dots,n-1\} : \xi_p\in
        N_{K_{n-1}/F}(K_{i}^\times) + N_{K/F}(K^\times)\}.
    \end{equation*}
    To see that $r\le a_0$, note that
    \begin{align*}
        \xi_p &= N_{K_{n-1}/F} (\delta_0) +
        N_{K/F}\left(\delta_1+\frac{d-1}{p}\alpha\right) \\ &\in
        N_{K_{n-1}/F}(K_{a_0}^\times)+N_{K/F}(K^\times).
    \end{align*}

    Now we show that $a_0\le r$.  Suppose that $\xi_p =
    N_{K_{n-1}/F}(\delta) + N_{K/F}(\gamma)$ for $\delta\in
    K_r^\times$ and $\gamma\in K^\times$. Then taking $p$-th powers
    and applying Hilbert 90, there exists $\alpha\in K^\times$ such
    that $(\sigma-1)\alpha = \delta+p\gamma$.  Hence for
    $\delta_0=\delta$ and $\delta_1=\gamma$ we see that $(\alpha,\{\delta_0,
    \delta_1\})$ represents the norm pair $((r),1)$.  Since
    $(\mathbf{a},d)\preceq_1 ((r),1)$, we deduce that $a_0\le r$.

    Therefore $a_0=r$.  From \cite[Theorem 3]{MSS} we obtain
    $i(K/F)=r$. Hence $\alpha$ is MSS-exceptional.  We have
    shown the first two items.

    For the third item, assume that $a_0\neq -\infty$.  Because $d\in U_1$ by
    definition of norm pair, $(\sigma-d)[\alpha]_1 = (\sigma-1)[\alpha]_1 =
    [\delta_0]_1$.  Since $\delta_0\in K_{a_0}^\times$, $\langle [\delta_0]_1\rangle$ is an
    $\Fp G_{a_0}$-module. Suppose that $\langle[\delta_0]_1\rangle$ is not a free $\Fp
    G_{a_0}$-module. Then by Lemma~\ref{le:idealrmgiII},
    $(\sigma-1)^{p^{a_0}-1}\in \ann_{\Fp G_{a_0}} \langle[\delta_0]_1\rangle$. Hence
    $(\sigma-1)^{p^{a_0}} [\alpha]_1 = 0$. Since the images of the
    powers of $\sigma-1$ on a generator of a cyclic $\Fp G_{a_0}$-module
    span the module over $\Fp$, $\langle[\alpha]_1\rangle$ has
    $\Fp$-dimension at most $p^{a_0}$.  On the other
    hand, by \cite[Theorem 2]{MSS}, $\langle [\alpha]_1\rangle$ has
    dimension $p^{a_0}+1$, a contradiction.
\end{proof}

\section{Norm Pairs Inequalities}\label{se:np2}
When we are in the case $\xi_p \in F$, and assuming $-1\in N_{K/F}(K^\times)$ if $p=2$ and $n=1$, it is more difficult to account for relations in $J_m$ between elements which form an $\F_pG$-basis of $J_1$.  To resolve this issue we will need to study minimal norm pairs and exceptional elements much more closely.  This section begins the investigation by establishing certain inequalities between elements of the norm vector in a minimal norm pair $(\mathbf{a},d)$.  Recall that for the remainder of the paper, we use additive notation for the operation in $K^\times$.

The following two results are the key inequalities we will use in later sections to uncover ``new" relations which appear in $J_m$.

\begin{proposition}\label{pr:basic.norm.pair.inequality}%\label{pr:twists}
Let $m\in \N$ and let $(\mathbf{a},d)$ be a minimal norm pair of
    length $m$.

    \begin{enumerate}
        \item\label{it:norm.pairs.are.increasing} If $a_i, a_j\neq -\infty$ for $0\le
        i<j<m$, then $a_i<a_j$.
        \item\label{it:p.2.d.3.mod.4.means.no.zero.ai} If $p=2$, $m\ge 2$, and $d\not\in U_2$, then $a_i\neq 0$ for any $1\le i<m$.
    \end{enumerate}
\end{proposition}

\begin{proposition}\label{pr:nv3}\
    Let $m\in \N$ and $(\mathbf{a},d)$ be a minimal norm pair of length $m$.

    \begin{enumerate}
%        \item\label{it:twn} If $d\in U_t\setminus U_{t+1}$ for some $1\le t<m$ and $a_t\neq -\infty$, then for each $i$ with $0\le i\le t$, we have $a_i< n-t+i$.
        \item\label{it:nv3w} We have
        \begin{equation*}
            a_{i}+j<a_{i+j}, \quad
            \begin{cases}
                0\le i< m, \\
                1\le j< m-i, \\
                a_i, a_{i+j}\neq -\infty.
            \end{cases}
        \end{equation*}
        except if $p=2$, $m\ge 2$, $d\not\in U_2$, and $i=a_i=0$.
        \item\label{it:nv3x} If $p>2$ or $d\in U_2$ and
        $d\in U_t\setminus U_{t+1}$ for $t<m$, or if $p = 2$ and $d \in -U_t\setminus -U_{t+1}$ for $t \geq 2$, then
        \begin{equation*}
            a_{t+k} > k, \quad \begin{cases}
                0\le k<m-t, \\
                a_{t+k}\neq -\infty.
            \end{cases}
        \end{equation*}
        \item\label{it:nv3y} If $p=2$, $a_0 = 0$ and $d \in -U_t \setminus -U_{t+1}$ for $2\le t$, then
        \begin{equation*}
            a_{t+k-1} > k, \quad \begin{cases}
                0\le k\leq m-t, \\
                a_{t+k-1}\neq -\infty.
            \end{cases}
        \end{equation*}
    \end{enumerate}
\end{proposition}

%\begin{proposition}\label{pr:droplower}
%    Suppose $m\in \N$ and let $(\mathbf{a},d)$ be a minimal norm pair
%    of length $m$.  Suppose that $d\not\in U_m$.  Then $d\in
%    U_{\omega}\setminus U_{\omega+1}$ and
%    \begin{enumerate}
%        \item if $p=2$ and $\omega=1$, then for $0\le k<n-\omega$, if
%        $a_{\omega+k}< k$, then $a_{\omega+k}=-\infty$
%        \item otherwise, then for $0\le k<n-\omega$, if
%        $a_{\omega+k}\le k$, then $a_{\omega+k}=-\infty$.
%    \end{enumerate}
%\end{proposition}

%We begin our proofs with the following Lemma.

\begin{proof}[Proof of Proposition~\ref{pr:basic.norm.pair.inequality}]
    (\ref{it:norm.pairs.are.increasing}). By Proposition~\ref{pr:norm.pair.extension.and.contraction} we may assume that $m=j+1$.  Suppose that $(\alpha,\{\delta_l\})$ represents $(\mathbf{a},d)$.  Then
    \begin{align*}
        (\sigma-d)\alpha &= \sum_{l=0}^m {p^l}\delta_l\\
        \xi_p &= \frac{d-1}{p}N_{K/F}(\alpha) +
        N_{K_{n-1}/F}(\delta_0) + \sum_{l=1}^{m}
        p^{l-1}N_{K/F}(\delta_l).
    \end{align*}

    Now suppose that for some $i$ and $j$ with $0\le i<j<m$, we have $a_i\ge a_j$ with $a_i, a_j\neq -\infty$. Set
    \begin{align*}
        \check \delta_l = \delta_l \quad (\text{for } 0\le l\le m, \
        l\neq i, j), \qquad 
        \check \delta_i = \delta_i+p^{j-i}\delta_j, \quad \text{ and }\quad
        \check \delta_j = 0.
    \end{align*}
    Then
    \begin{equation*}
        (\sigma-d)\alpha=\sum_{l=0}^m
        p^l \check \delta_l.
    \end{equation*}
    If $i=0$ then we have $a_0<n$, which in turn implies $a_j<n$ as well. Hence we have
    \begin{align*}
        N_{K_{n-1}/F}(\check \delta_0) &=
        N_{K_{n-1}/F}(\delta_0) + p^j N_{K_{n-1}/F}(\delta_j) \\
        &= N_{K_{n-1}/F}(\delta_0) + p^{j-1} N_{K/F}(\delta_j).
    \end{align*}
    Otherwise, since $a_i\ge a_j$ we have
    \begin{align*}
        {p^{i-1}}N_{K/F}(\check \delta_i) &=
        {p^{i-1}}N_{K/F}(\delta_i)+
        {p^{i-1+j-i}}N_{K/F}(\delta_j) \\
        &= {p^{i-1}} N_{K/F}(\delta_i)+{p^{j-1}}N_{K/F}(\delta_j).
    \end{align*}
    Hence
    \begin{equation*}
        \xi_p = {\frac{d-1}{p}}N_{K/F}(\alpha)+
        N_{K_{n-1}/F}(\check\delta_0) + \sum_{l=1}^{m}
        {p^{l-1}}N_{K/F}(\check\delta_l).
    \end{equation*}
    Now since $a_i\ge a_j$, $\check\delta_i\in K_{a_i}^\times$. Then since $\check\delta_j=1$, we see that $(\alpha,\{\check\delta_l\})$ represents
    \begin{equation*}
        ((a_0,\dots,a_{j-1},-\infty),d).
    \end{equation*}
    Since $(\mathbf{a},d)\preceq_m ((a_0,\dots,a_{j-1},-\infty),d)$, we have reached a contradiction with $a_j\neq -\infty$.  Hence $a_i<a_j$ for $0\le i<j<m$ and $a_i$, $a_j\neq -\infty$.

    (\ref{it:p.2.d.3.mod.4.means.no.zero.ai}).  Suppose that $p=2$, $m\ge 2$, $d\not\in U_2$, and $a_i=0$ for some $1\le i< m$.  By Proposition~\ref{pr:norm.pair.extension.and.contraction} we may assume that $m=i+1$, and by item~(\ref{it:norm.pairs.are.increasing}) we may assume that $a_j=-\infty$ for all $0\le j<i$.  Let $(\alpha,\{\delta_l\})$ represent $(\mathbf{a},d)$, and write $d=-1+4x$ for some $x\in \Z$.

    Then
    \begin{align*}
        (\sigma-d)\alpha &=  \sum_{l=0}^m {2^l}\delta_l
        = 2^i \delta_i+2^m \delta_{m}% \\
%        &= \alpha^{-1}\alpha^{4x}
%        \delta_i^{2^i} \delta_m^{2^m}.
    \end{align*}
    and
    \begin{align*}
        -1 &= {\frac{d-1}{2}}N_{K/F}(\alpha)+
        N_{K_{n-1}/F}(\delta_0) + \sum_{l=1}^{m}
        {2^{l-1}}N_{K/F}(\delta_l) \\
        &= (-1+2x)N_{K/F}(\alpha)+ {2^{i-1}}N_{K/F}(\delta_i)+
        {2^{m-1}}N_{K/F}(\delta_m).
    \end{align*}

    Set
    \begin{align*}
        \check\alpha = \alpha + 2^{i-1}\delta_i \quad \quad 
        \check\delta_0 = \cdots = \check \delta_{m-1} = 0\quad \quad 
        \check\delta_{m} = (1-x)\delta_i+\delta_m.
    \end{align*}
    Then, recalling that $\delta_i\in K_{a_i}^\times=F^\times$ and
    $m=i+1$,
    \begin{align*}
        (\sigma-d)\check\alpha &= (\sigma-d)(\alpha+2^{i-1}\delta_i) \\
        &= 2^i\delta_i + 2^m \delta_m + 2^{i-1}\delta_i - (-1+4x)2^{i-1}\delta_i \\
%        &= 2^i\delta_i + 2^m \delta_m + 2^i\delta_i-x2^{i+1}\delta_i\\
        &= 2^m \check \delta_m.
    \end{align*}
    Moreover, using the fact that $\frac{d-1}{2} = -1+2x$, we recover
    \begin{align*}
        \frac{d-1}{2}&N_{K/F}(\check\alpha)
        +2^{m-1}N_{K/F}(\check\delta_m)\\&={\frac{d-1}{2}}N_{K/F}(\alpha+2^{i-1}\delta_i)+2^{m-1}N_{K/F}((1-x)\delta_i+\delta_m)\\ 
%        &= {\frac{d-1}{2}}N_{K/F}(\alpha)+
%        ({2^i-2^{i-1}})N_{K/F}(\delta_i) +{2^{m-1}}N_{K/F}(\delta_m)\\ 
        &= {\frac{d-1}{2}}N_{K/F}(\alpha)+
        {2^{i-1}}N_{K/F}(\delta_i)+{2^{m-1}} N_{K/F}(\delta_m)=-1.%\\ &= -1.
    \end{align*}
    Hence $(\check\alpha,\{\check\delta_l\})$ represents
    \begin{equation*}
        ((-\infty,\dots,-\infty),d).
    \end{equation*}
    Since $(\mathbf{a},d)\preceq_m ((-\infty,\dots,-\infty),d)$ and
    $a_i\neq -\infty$, we have reached a contradiction.
\end{proof}
Before proving Proposition \ref{pr:nv3} we will need some preliminary results.  
The following lemma will be of crucial importance in establishing properties of minimal norm pairs: it shows that, under certain conditions, we need only show one of the two halves of the norm pair condition.

\begin{lemma}\label{le:excmod}
Let $m\in \N$.  Suppose that $(\mathbf{a},d)$ is a norm pair
    of length $m$ represented by $(\alpha,\{\delta_i\})$ where
    $\alpha$ is $1$-exceptional.

    Suppose that $\check\alpha\in K^\times$ satisfies
    $[\check\alpha]_1-[\alpha]_1\in [K_{n-1}^\times]_1 + (\sigma-1)J_1$ and
    \begin{equation*}
        (\sigma-\check d)\check\alpha = \sum_{i=0}^m
        p^i\check\delta_i
    \end{equation*}
    with $\check\delta_i\in K_{a_i}^\times$, $0\le i<m$, for $\check d\in U_1$.
    Then $(\check\alpha,\{\check\delta_i\})$ represents the norm pair
    $(\mathbf{a},\check d)$.
\end{lemma}

\begin{proof}
    Since $\alpha$ is $1$-exceptional, Proposition~\ref{pr:excprop}
    tells us that
    \begin{equation*}
        \frp N_{K/F}(\alpha) \in K^\times\quad \text{and}\quad
        (\sigma-1) \frp N_{K/F}(\alpha) = \xi_p.
    \end{equation*}
    Since $[\check\alpha]_1 - [\alpha]_1\in [K_{n-1}^\times]_1+
    (\sigma-1)J_1$, we see that
    $N_{K/F}(\check\alpha-\alpha)\in pF^\times$ and so for some $f \in F^\times$ we have
    \begin{equation*}
        \frp N_{K/F}(\check \alpha)= \frp (N_{K/F}(\alpha) +pf) \in
        K^\times.
    \end{equation*}
    Furthermore
    \begin{equation}\label{eq:excmodx}
        (\sigma-1) \frp N_{K/F}(\check \alpha) = (\sigma-1) \frp
        N_{K/F}(\alpha) = \xi_p.
    \end{equation}

    Set $\check\gamma=\check\delta_0$ and $\check\delta=\sum_{i=1}^{m} p^{i-1}\check\delta_i+\frac{\check d-1}{p}\check\alpha$.  Then $(\sigma-1)\check\alpha = \check\gamma+p\check\delta$.

    Now if we are not in the case $p=2$ and $n=1$, by
    Lemma~\ref{le:equivalent.ways.to.express.power.of.xi.as.product.of.norms},
    \begin{equation*}
        \xi_p = \frac{\check d-1}{p}N_{K/F}(\check\alpha) +
        N_{K_{n-1}/F}(\check\delta_0) + \sum_{i=1}^{m}
        p^{i-1}N_{K/F}(\check\delta_i).
    \end{equation*}
    Hence $(\check\alpha,\{\check\delta_i\})$ represents $(\bar
    a,\check d)$.

    If $p=2$ and $n=1$ then by
    Lemma \ref{lem:p.2.and.n.1.gives.a0.minus.infinity}, $a_0=-\infty$.  Hence $\check\delta_0=0$ and so $(\sigma-1)\check \alpha = 2 \check \delta$ for $\check \delta := \sum_{i=1}^m 2^{i-1} \check \delta_i + \frac{\check d-1}{2}\check\alpha$.  Also from \eqref{eq:excmodx} we have
    \begin{equation*}
        (\sigma-1)\frt N_{K/F}(\check\alpha) =
        (\sigma-1)\frt N_{K/F}(\alpha) = -1.
    \end{equation*}
    The action of
    $\sigma-1$ on the two square roots of $N_{K/F}(\check\alpha) =
    (\sigma+1)\check \alpha = 2\check\delta + 2\check\alpha$ are
    identical, and we calculate the action for one:
    \begin{align*}
        -1=(\sigma-1)\frt N_{K/F}(\check\alpha) &=
        (\sigma-1)(\check\delta + \check\alpha) \\ &=
        (\sigma-1)\check\delta  + 2\check\delta \\ &=
        N_{K/F}(\check\delta).
    \end{align*}
    Computing $N_{K/F}(\check\delta)$ out explicitly, this last equality gives
    \begin{equation*}
        -1 = \frac{\check d-1}{2}N_{K/F}(\check\alpha) +
        \sum_{i=1}^{m} p^{i-1}N_{K/F}(\check\delta_i).
    \end{equation*}
    Hence $(\check\alpha,\{\check\delta_i\})$ represents $(\bar
    a,\check d)$.
\end{proof}

\begin{lemma}\label{le:delicate}
    Let $m\ge 2$.
    Suppose that $(\mathbf{a},d)$ is a norm pair
    of length $m$ represented by $(\alpha,\{\delta_i\})$.  Suppose moreover that for
    some $i$ with $0\le i<m-1$, $a_i\neq -\infty$, and $0<a_{m-1}-a_i\le m-1-i$.  Finally, suppose that either $p>2$, $a_i>0$, or $d\in U_2$. Then there exist $1$-exceptional $\check\alpha$ and $\{\check\delta_i\}$ such that $(\alpha,\{\check\delta_i\})$ represents $((a_0,\dots,a_{m-2},-\infty),d)$.
\end{lemma}

\begin{proof}
    Assume the hypotheses of the lemma, and to simplify notation set $t=m-1$.
    By Lemma~\ref{le:phiII}(\ref{it:phi2II}),
    \begin{equation*}
        \phi_d(P(a_{t},a_i)) = p^{a_{t}-a_i} \pmod{p^{a_{t}-a_i+1}}.
    \end{equation*}
    Since $t-i\ge a_{t}-a_i$, we have
    $p^{t-(a_{t}-a_i)}\in \Z$, and multiplying by it we find
    \begin{equation*}
        \phi_d(p^{t-(a_{t}-a_i)}P(a_{t},a_i)) = p^{t} \pmod{p^{t+1}}.
    \end{equation*}
    Hence there exists $Q\in \Z G$ such that
    \begin{equation*}
        (\sigma-d)Q = p^{t-(a_{t}-a_i)}P(a_{t},a_i)-p^{t}
        \pmod{p^{t+1}\Z G}.
    \end{equation*}

    We claim that
    \begin{equation}\label{eq:whenQisinI}
    \mbox{if }a_{t}=n \mbox{ then } Q\in \langle
    (\sigma-1),p\rangle\subset \Z G,
    \end{equation}
    as follows.  We first reduce
    the previous equation modulo $p \Z G$ so that we obtain an
    equation in $\Fp G$.  Observing that $t\ge 1$,
    $P(a_{t},a_i)=(\sigma-1)^{p^{a_{t}}-p^{a_i}}$ in $\Fp G$, and, by the definition of norm pair, $d\in U_1$,
    we have
    \begin{equation*}
        (\sigma-1)Q =
        p^{t-(a_{t}-a_i)}(\sigma-1)^{p^{a_{t}}-p^{a_i}} \qquad
        \text{in\ } \Fp G.
    \end{equation*}
    Now suppose that $a_{t}=n$ and $Q\not\in \langle (\sigma-1), p\rangle \subset \Z G$.  Then $Q$ reduces to a unit in $\Fp G$. If $t> a_{t}-a_i$, we have a contradiction: $\langle\sigma-1\rangle \neq \{0\}$ in $\Fp G$.  Otherwise, since $t-i\ge a_{t}-a_i$ we obtain $t-i=a_{t}-a_i$ and $i=0$. We then have
    \begin{equation*}
        (\sigma-1)Q = (\sigma-1)^{p^{a_{t}}-p^{a_i}} \qquad
        \text{in\ } \Fp G.
    \end{equation*}
    Considering the ideals generated by $\left \{(\sigma-1)^l:l\in \{1,\cdots,p^n\}\right\}$ are distinct in $\Fp G$ and the fact that $Q$ is a unit in $\Fp G$, we deduce that $p^{a_{t}}-p^{a_i}=1$.  Since $a_i, a_{t}\neq -\infty$, we have $a_i=0$, $a_{t}=1$, and $p=2$.  We already determined that $i=0$, so $a_0=0$.  But then $p=2$, $n=1$ and $a_0=0\neq -\infty$, contradicting Lemma \ref{lem:p.2.and.n.1.gives.a0.minus.infinity}.  Hence if $a_{t}=n$ we have $Q\in \langle (\sigma-1),p\rangle$.

    Now by Corollary~\ref{cor:contacting.exceptional.elements}, $\alpha$ is
    $1$-exceptional.  Let $T\in \Z G$ satisfy
    \begin{equation*}
        (\sigma-d)Q = p^{t-(a_{t}-a_i)}P(a_{t},a_i)-p^{t} +
        p^{t+1}T.
    \end{equation*}
    Set
    \begin{align*}
        \check \alpha = \alpha + Q\delta_{t} &\qquad\qquad 
        \check \delta_l = \delta_l \quad(\text{for } 0\le l< m, \
        l\neq i, t\\
        \check \delta_{t} = 0 & \qquad\qquad
        \check \delta_i = \delta_i +
        p^{t-i-(a_{t}-a_i)}P(a_{t},a_i)\delta_{t} \\ 
        \check \delta_{m} = \delta_m + T\delta_{t}&\qquad\qquad \quad = \delta_i +
        p^{t-i-(a_{t}-a_i)}N_{K_{a_{t}}/K_{a_i}}(\delta_{t}).
    \end{align*}
    Since
    \begin{equation*}
        (\sigma-d)Q\delta_{t} =
        p^{t-(a_{t}-a_i)}P(a_{t},a_i)\delta_{t}-p^{t}\delta_{t}
        +p^{t+1}T\delta_{t}
    \end{equation*}
    and recalling that $m=t+1$, we obtain
    \begin{equation*}
        (\sigma-d)\check \alpha = \sum_{l=0}^{m} p^l \check
        \delta_l.
    \end{equation*}

    Since $Q\delta_{t}\in K_{n-1}^\times$ if $a_{t}<n$ and statement (\ref{eq:whenQisinI}) gives $Q\delta_{t} \in (\sigma-1)K^\times + pK^\times$ otherwise, we may invoke Lemma~\ref{le:excmod} and obtain that $(\check\alpha,\{\check\delta_l\})$ represents $(\mathbf{a}, d)$.  But since $\check\delta_{t}=0$ we have moreover that $(\check\alpha,\{\check\delta_l\})$ represents $((a_0,\dots,a_{t-1},-\infty),d)$, as desired.
\end{proof}
We are now ready to prove Proposition~\ref{pr:nv3}.
\begin{proof}[Proof of Proposition~\ref{pr:nv3}]
    We derive each reduction statement by altering $\alpha$ in the fashion of Lemma~\ref{le:excmod}, using either Lemma~\ref{le:delicate} or a relation obtained from Lemma~\ref{le:upowerII} or Lemma~\ref{le:phiII}.

    (\ref{it:nv3w}).  Suppose that $0\le i<m$, $1\le j<m-i$, $a_i$, $a_{i+j}\neq -\infty$, and $a_{i+j}\le a_i+j$.  We will derive a contradiction.  By Proposition~\ref{pr:norm.pair.extension.and.contraction} we may assume that $m=i+j+1$ and we have $m\ge 2$.   Suppose also that we are not in the case $p=2$, $d\not\in U_2$, $i=a_i=0$.  Then by Proposition~\ref{pr:basic.norm.pair.inequality}(\ref{it:p.2.d.3.mod.4.means.no.zero.ai}), if $p=2$ and $d\not\in U_2$ then $a_i\neq 0$.  Hence we have at least one of $p>2$, $d\in U_2$, or $a_i> 0$.

    Now let $t=i+j$.  Then by Lemma~\ref{le:delicate}, we obtain the existsnce of $(\check\alpha,\{\check\delta_l\})$ representing $((a_0,\dots,a_{i+j-1},-\infty),d)$.  But then since $(\mathbf{a},d) \preceq_m ((a_0,\dots,a_{i+j-1},-\infty),d)$, we have reached a contradiction of $a_{i+j}\neq -\infty$, as desired.

    (\ref{it:nv3x}).  Suppose that $d\in U_t\setminus U_{t+1}$ for some $t$ with $1\le t<m$ if either $p>2$ or $d\in U_2$, or that $p=2$ and $d \in -U_t\setminus -U_{t+1}$ for some $w \geq 2$.  Suppose moreover that for some $0\le k<m-t$, $a_{t+k}\le k$ and $a_{t+k}\neq -\infty$.  We will derive a contradiction.  By Proposition~\ref{pr:norm.pair.extension.and.contraction} we may assume that $m=t+k+1$.  Furthermore if $p=2$ and $d \not\in U_2$, then $a_{t+k}\neq 0$ by Proposition \ref{pr:basic.norm.pair.inequality}; hence by Lemma~\ref{le:upowerII},
    \begin{equation*}
        d^{p^{a_{t+k}}} -1 = xp^{t+a_{t+k}}
    \end{equation*}
    for some $x\in \Z\setminus p\Z$.  Let $y$ satisfy $xy=1 \pmod{p^m}$.  Since $k\ge a_{t+k}$, we have that $yp^{k-a_{t+k}}\in \Z$, and, multiplying by it, we find
    \begin{equation*}
        yp^{k-a_{t+k}}(d^{p^{a_{t+k}}}-1) = p^{t+k} \pmod{p^m}.
    \end{equation*}
    Hence
    \begin{equation*}
        \phi_d(yp^{k-a_{t+k}}(\sigma^{p^{a_{t+k}}}-1) - p^{t+k})
        = 0 \pmod{p^m}.
    \end{equation*}
    Hence there exist $Q, T\in \Z G$ satisfying
    \begin{equation*}
        (\sigma-d)Q = yp^{k-a_{t+k}}(\sigma^{p^{a_{t+k}}}-1) - p^{t+k} + p^mT.
    \end{equation*}

Observe that if $a_{t+k}=n$, then the previous equation considered in $\F_pG$ is $(\sigma-d)Q=0$.  Hence we have $Q \in \langle \sigma-1,p\rangle \subseteq \Z G$ in this case.

%    We claim that $a_{t+k}<n$, as follows.  Observe that if $a_{t+k}=n$ then by (\ref{it:twn}), $d\in U_m$.  Now if $p>2$ or $d \in U_2$ this contradicts the assumption that $d \in U_t \setminus U_{t+1}$ where $t<m$.  Otherwise we are in the case $p=2$ and $d \in -U_t \setminus -U_{t+1}$, and since $0\le k<m-t$ and $t\ge 2$, we have $m\ge 3$, in which case $d \in U_m$ is a clear contradiction. Hence in either case we have $a_{t+k}<n$.

    Now suppose that $(\alpha,\{\delta_l\})$ represents $(\mathbf{a},d)$.  By Corollary~\ref{cor:contacting.exceptional.elements}, we have that $\alpha$ is
    $1$-exceptional.
    Set
    \begin{align*}
        \check \alpha &= \alpha + Q\delta_{t+k} \\
        \check \delta_l &= \delta_l, \qquad\qquad\qquad 0\le l< m, \
        l\neq t+k\\
        \check \delta_{t+k} &= 0 \\
        \check \delta_{m} &= \delta_m + T\delta_{t+k}.
    \end{align*}
    Since $\delta_{t+k}\in K_{a_{t+k}}^\times$,
    $(\sigma^{p^{a_{t+k}}}-1)\delta_{t+k}=0$.  From
    \begin{align*}
        (\sigma-d)Q\delta_{t+k} &=
        yp^{k-a_{t+k}}(\sigma^{p^{a_{t+k}}}-1)\delta_{t+k} -
        p^{t+k}\delta_{t+k} + p^mT\delta_{t+k} \\ &=
        -p^{t+k}\delta_{t+k} + p^mT\delta_{t+k}
    \end{align*}
    we obtain
    \begin{equation*}
        (\sigma-d)\check \alpha = \sum_{l=0}^{m} p^l \check
        \delta_l.
    \end{equation*}

    Since $Q\delta_{t+k}\in K_{n-1}^\times$ if $a_{t+k} <n$, and $Q\delta_{t+k} \in (\sigma-1)K^\times + pK^\times$ otherwise, we may invoke
    Lemma~\ref{le:excmod} and obtain, as before, that
    $(\check\alpha,\{\check\delta_l\})$ represents
    $((a_0,\dots,a_{t+k-1},-\infty),d)$.  Since
    \begin{equation*}
        (\mathbf{a},d) \preceq_m
        ((a_0,\dots,a_{t+k-1},-\infty),d),
    \end{equation*}
    we have reached a contradiction of $a_{t+k}\neq -\infty$, as desired.

    (\ref{it:nv3y}).  We proceed as in Lemma~\ref{le:delicate}.  Assume that $p=2$ and $d\in -U_t\setminus -U_{t+1}$ for some $2\le t$.  Assume moreover that for some $0\le k\le m-t$, $a_{t+k-1}\le k$ and $a_{t+k-1}\neq -\infty$.  We will derive a contradiction.  By Proposition~\ref{pr:norm.pair.extension.and.contraction} we may assume that $m=t+k$. 

    Note that $t+k-1 >0$, and hence by Proposition \ref{pr:basic.norm.pair.inequality}(\ref{it:p.2.d.3.mod.4.means.no.zero.ai}) we have $a_{t+k-1}>0$.  Therefore by Lemma~\ref{le:phiII}(\ref{it:phi3II}),
    \begin{equation*}
        \phi_d(P(a_{t+k-1},0)) = 2^{a_{t+k-1}+t-1} \pmod{
        2^{a_{t+k-1}+t}}.
    \end{equation*}
    Since $a_{t+k-1} \le k$, we may multiply by
    $2^{k-a_{t+k-1}}$ to find
    \begin{equation*}
        \phi_d(2^{k-a_{t+k-1}}P(a_{t+k-1},0)) = 2^{t+k-1} \pmod{
        2^{t+k}}.
    \end{equation*}
    Hence there exists $Q\in \Z G$ such that
    \begin{equation*}
        (\sigma-d)Q = 2^{k-a_{t+k-1}}P(a_{t+k-1},0)-2^{t+k-1}
        \pmod{2^{t+k}\Z G}.
    \end{equation*}

    We claim that if $a_{t+k-1}=n$ then $Q\in \langle (\sigma-1),2\rangle\subset \Z G$, as follows.  We first reduce the previous equation modulo $2 \Z G$ so that we obtain an equation in $\Ft G$.  Observing that     $P(a_{t+k-1},0)=(\sigma-1)^{2^{a_{t+k-1}}-1}$ in $\Ft G$, $t+k-1\ge 1$, and, by the definition of norm pair, $d\in U_1$, we have
    \begin{equation*}
        (\sigma-1)Q =
        2^{k-a_{t+k-1}}(\sigma-1)^{2^{a_{t+k-1}}-1} \qquad
        \text{in\ } \Ft G.
    \end{equation*}
    Now suppose that $a_{t+k-1}=n$ and $Q\not\in \langle (\sigma-1), 2\rangle \subset \Z G$.  Then $Q$ reduces to a unit in $\Ft G$. If $k> a_{t+k-1}$, we have a contradiction: $\langle\sigma-1\rangle \neq \{0\}$ in $\Ft G$.  Hence $k=a_{t+k-1}$.  Considering
    \begin{equation*}
        (\sigma-1)Q = (\sigma-1)^{2^{a_{t+k-1}}-1} \qquad
        \text{in\ } \Ft G,
    \end{equation*}
    alongside the fact that the ideals generated by $\left\{(\sigma-1)^l: l\in\{1, \cdots, 2^n\}\right\}$ are distinct in $\Ft G$ and that $Q$ is a unit in $\Ft G$, we deduce that $2^{a_{t+k-1}}-1=1$ and so $a_{t+k-1}=1$.  But then $p=2$, $n=a_{t+k-1}=1$ and $a_0=0\neq -\infty$, contradicting Lemma \ref{lem:p.2.and.n.1.gives.a0.minus.infinity}.  Hence if $a_{t+k-1}=n$ we have $Q\in \langle (\sigma-1),2\rangle$.

    Now suppose that $(\alpha,\{\delta_l\})$ represents $(\mathbf{a},d)$.  By Corollary~\ref{cor:contacting.exceptional.elements}, we have that $\alpha$ is
    $1$-exceptional.  Let $T\in \Z G$ satisfy
    \begin{equation*}
        (\sigma-d)Q = 2^{k-a_{t+k-1}}P(a_{t+k-1},0)-2^{t+k-1} +
        2^{t+k}T.
    \end{equation*}
    Set
    \begin{align*}
        \check \alpha &= \alpha + Q\delta_{t+k-1} \\
        \check \delta_l &= \delta_l, \qquad\qquad\qquad 0\le l< m, \
        l\neq 0, t+k-1\\
        \check \delta_0 &= \delta_0 +
        2^{k-a_{t+k-1}}P(a_{t+k-1},0)\delta_{t+k-1} \\ &= \delta_0 +
        2^{k-a_{t+k-1}}N_{K_{a_{t+k-1}}/F}(\delta_{t+k-1})\\
        \check \delta_{t+k-1} &= 0 \\
        \check \delta_{m} &= \delta_m + T\delta_{t+k-1}.
    \end{align*}
    Since
    \begin{equation*}
        (\sigma-d)Q\delta_{t+k-1} =
        2^{k-a_{t+k-1}}P(a_{t+k-1},0)
        \delta_{t+k-1}-2^{t+k-1}\delta_{t+k-1}
        +2^{t+k}T\delta_{t+k-1}
    \end{equation*}
    we obtain
    \begin{equation*}
        (\sigma-d)\check \alpha = \sum_{l=0}^{m} 2^l \check
        \delta_l.
    \end{equation*}

    Since $Q\delta_{v+k-1}\in K_{n-1}^\times$ if $a_{t+k-1}<n$ and $Q\delta_{t+k-1} \in (\sigma-1)K^\times + 2K^\times$ otherwise, we may invoke Lemma~\ref{le:excmod} and obtain, as before, that
    $(\check\alpha,\{\check\delta_l\})$ represents
    $((a_0,\dots,a_{t+k-2},-\infty),d)$.  Since
    \begin{equation*}
        (\mathbf{a},d) \preceq_m
        ((a_0,\dots,a_{t+k-2},-\infty),d),
    \end{equation*}
    we have reached a contradiction of $a_{t+k-1}\neq -\infty$, as
    desired.
\end{proof}

\section{Exceptional Elements}\label{se:excelts}

In this section we show that elements $(\alpha,\{\delta_i\})$ which represent a minimal norm pair satisfy some basic module-theoretic conditions in $J_1$: each $\delta_i \in K_{a_i}^\times$ generates a free $\F_pG_{a_i}$-module (Proposition \ref{pr:deltafree}), and $\{\alpha,\delta_1,\cdots,\delta_{m-1}\}$ can be chosen as a part of an $\F_pG$-basis for $J_1$ (Proposition \ref{prop:picking.the.right.basis}).  As we'll see in section \ref{sec:exceptional.modules}, these results  play a critical role in controlling the appearance of ``new" dependencies in $J_m$, and hence its overall module structure.

\begin{proposition}\label{pr:deltafree}
Suppose $m\in \N$ and $(\alpha,\{\delta_i\})$ represents a
    minimal norm pair $(\mathbf{a},d)$.   Then for
    $0\le i<m$ and $a_i\neq -\infty$, the $\Fp G_{a_i}$-module $\langle
    [\delta_i]_1\rangle$ is free.
\end{proposition}

\begin{proof}
By Corollary~\ref{cor:contacting.exceptional.elements}, $m$-exceptional elements are
    $1$-exceptional.  We consider the case $i=0$ first.  Since $d\in
    U_1$, $$(\sigma-d)[\alpha]_1 = (\sigma-1)[\alpha]_1 = [\delta_0]_1,$$ and then
    the case $i=0$ follows directly from Proposition~\ref{pr:MSS}(\ref{it:MSS3}).

    Assume then that $i>0$ and $a_i\ge 0$, and assume the other
    hypotheses of the proposition.  By Proposition~\ref{pr:norm.pair.extension.and.contraction}, we
    may assume that $m=i+1$.

%    If $a_0=-\infty$ then certainly $a_i>a_0$, and by
%    Proposition~\ref{pr:basic.norm.pair.inequality}(\ref{it:norm.pairs.are.increasing}), if $a_0\neq
%    -\infty$ then $a_i>a_0$.

%    By Lemma~\ref{le:MSSstruct}, if $a_i < n$ then
%    \begin{equation*}
%        [K_{a_i}^\times]_1 = \sum_{l=0}^n Y_l
%    \end{equation*}
%    where
%    \begin{enumerate}
%        \item $Y_l\subset [K_l^\times]_1$ for $l\le a_i$,
%        \item $Y_l$ is a free $\Fp G_l$-module for $l < a_i$
%        \item $Y_l$ is a free $\Fp G_{a_i}$-module for $l\ge a_i$
%        \item $Y_l\subset (\sigma-d)J_1$ for $l>a_i$.
%    \end{enumerate}
    By \cite[Th.~2]{MSS} we have $$J_1= \langle [\alpha]_1 \rangle \oplus \bigoplus_{i=0}^n Y_l,$$ where each $Y_l \subset [K_l^\times]_1$ is a free $\F_pG_l$-module.  Let $\Ycc_l\subset
    K_{l}^\times$ be chosen such the elements form an $\Fp
    G_{l}$-base for $Y_l$ modulo $pK^\times$: $\bigoplus_{\gamma\in \Ycc_l} \langle[\gamma]_1\rangle=Y_l$.  Observe that if $\gamma\in \Ycc_l$, then
    $\langle [\gamma]_1\rangle$ is a free $\Fp
    G_{l}$-module.

	We then have 
    \begin{equation*}
        \Ycc = \{\alpha\} \cup \bigcup_{l=0}^n \Ycc_l
    \end{equation*}
    is an $\Fp G$-base for $J_1$.  For $0\le l\le n$, write
    $\Ycc_{\le l}$ for $\Ycc_0 \cup \cdots \cup \Ycc_l$, and for
    $0\le l<n$, write $\Ycc_{>l}$ for $\Ycc_{l+1}\cup \cdots \cup
    \Ycc_n$.

%    Now we are given that $(\sigma-d)\alpha = \sum_{k=0}^{m} p^k\delta_k$
%    with $\delta_k\in K_{a_k}^\times$ for all $0\le k<m$.

    Express $[\delta_i]_1$ as a combination from $\Ycc$:
    \begin{equation*}
        [\delta_i]_1 = g_\alpha [\alpha]_1 + \sum_{\gamma\in \Ycc} g_\gamma
        [\gamma]_1, \quad \quad g_\alpha,g_\gamma \in \Z G,
    \end{equation*}
    with almost all $g_\gamma=0$. For
    each $g_\gamma$ with $\gamma\in\Ycc_{\le a_i}$ write
    \begin{equation*}
        g_\gamma=c_\gamma+(\sigma-d)e_\gamma  \pmod{p\Z G}, \qquad
        c_\gamma\in \{0, \cdots, p-1\},\ e_\gamma\in \Z G.
    \end{equation*}
    If $a_i<n$, then note that it must be the case that $g_\gamma \in \langle(\sigma-d)\rangle$ for all $\gamma \in \mathcal{Y}_{>a_i}$; if this were not the case, then $g_\gamma$ would be a unit for some $\gamma \in \mathcal{Y}_{>a_i}$, and hence $(\sigma-1)^{p^{a_i}}g_\gamma[\gamma]_1 \neq [0]_1$.  But then we'd have $(\sigma^{p^{a_i}}-1) = (\sigma-1)^{p^{a_i}} \not\in \ann_{\F_pG} [\delta_i]_1$, contrary to the fact that $\langle[\delta_i]_1\rangle$ is an $\F_pG_{a_i}$-module.  So for each $g_\gamma$ with $\gamma\in
    \Ycc_{>a_i}$, let $\gamma'\in K^\times$ satisfy
    $(\sigma-d)[\gamma']_1 = [\gamma]_1$.    
    
    Hence
    \begin{equation*}
        [\delta_i]_1 = g_\alpha [\alpha]_1 + \sum_{\gamma\in \Ycc_{\le a_i}}
        c_\gamma [\gamma]_1 + (\sigma-d)\left(\sum_{\gamma\in \Ycc_{\le
        a_i}} e_\gamma[\gamma]_1+\sum_{\gamma\in \Ycc_{>a_i}}
        g_\gamma[\gamma']_1\right).
    \end{equation*}
 Let
 \begin{equation*}
        \varepsilon = \sum_{\gamma\in \Ycc_{\le a_i}} e_\gamma\gamma+
        \sum_{\gamma\in \Ycc_{>a_i}} g_\gamma\gamma'
    \end{equation*}
    and observe that
    \begin{equation*}
        (\sigma-d)\varepsilon = \delta_i-g_\alpha \alpha - \sum_{\gamma
        \in \Ycc_{\le a_i}} c_\gamma \gamma + pk
    \end{equation*}
    for some $k \in K^\times$.
    Now set
    \begin{equation*}
\begin{array}{rlcrl}
        \check \alpha &= \alpha - p^i\varepsilon &\quad \quad & 
        \check d &= d+p^ig_\alpha,\\
		\check \delta_k &= \delta_k \quad \text{ for }\quad 0 \leq k < i \quad &&
		\check \delta_i &= \sum_{\gamma \in \mathcal{Y}_{\leq a_i}}c_\gamma \gamma\\
	\check \delta_{i+1} &= \delta_{i+1} - k + p^{i-1} g_\alpha\varepsilon.&&&
    \end{array}
    \end{equation*}

    Since $i\ge 1$ we have $[\alpha]_1=[\check\alpha]_1$.  Moreover, we calculate
    \begin{align*}
        (\sigma-\check d)&\check \alpha = (\sigma-d)\check\alpha
        - p^ig_\alpha\check\alpha = (\sigma-d)\alpha - p^i(\sigma-d)
        \varepsilon - p^ig_\alpha\alpha + p^{2i}g_\alpha\varepsilon \\
        &= (\sigma-d)\alpha - p^i\delta_i + p^ig_\alpha\alpha
        + p^i\sum_{\gamma\in \Ycc_{\le a_i}} c_\gamma\gamma -p^{i+1}k -p^ig_\alpha\alpha
        +p^{2i}g_\alpha \varepsilon\\
        &= (\sigma-d)\alpha -p^i\delta_i + p^i\sum_{\gamma\in \Ycc_{\le
            a_i}}c_\gamma\gamma - p^{i+1}k + p^{2i}g_\alpha\varepsilon\\
	&= \sum_{k=0}^{i+1} p^k\check \delta_k
    \end{align*}
    By construction $\check\delta_k\in K_{a_k}^\times$
    for all $0\le k\le i$.  Hence by Lemma~\ref{le:excmod},
    $(\check\alpha, \{\check\delta_k\})$ represents $(\mathbf{a},
    \check d)$.

    Now suppose $c_\gamma=0$ for all $\gamma\in \Ycc_{a_i}$.  This implies that $\check \delta_i \in K_{a_i-1}^\times$, so that $(\check \alpha,\{\check \delta_k\})$ represents the norm pair $((a_0, \cdots, a_{i-1}, a_i-1),\check d)$.  This contradicts the minimality of $(\mathbf{a}, d)$, and hence $c_\gamma \neq 0$ for some $\gamma \in \mathcal{Y}_{a_i}$.

    Now since $\delta_i\in K_{a_i}^\times$ we have $\langle
    [\delta_i]_1\rangle$ is an $\Fp G_{a_i}$-module. Now $\langle
    [\gamma]_1\rangle$ generates a free $\Fp G_{a_i}$-module which is
    a direct summand of $Y_{a_i}$, which is itself a direct summand
    of $J_1$.  Since $c_\gamma \neq 0$, we have  $c_\gamma$ is unit in
    $\Fp G_{a_i}$. Hence $\ann_{\Fp G_{a_i}} c_\gamma[\gamma]_1=\{0\}$,
    and then $\langle[\delta_i]_1\rangle$ is a free $\Fp
    G_{a_i}$-module as well.
\end{proof}

\begin{proposition}\label{prop:picking.the.right.basis}
Suppose $m \in \N$ and $(\alpha,\{\delta_i\}_{i=0}^{m-1})$ represents a minimal norm pair $(\mathbf{a},d)$ of length $m$. Write $\mathcal{X}_m = \{\alpha,\delta_1,\cdots,\delta_{m-1}\}\setminus\{0\}$.  Then there exists $\mathcal{T} \subset K^\times$ so that $$J_1 = \bigoplus_{x \in \mathcal{X}_m} \langle [x]_1\rangle \oplus \bigoplus_{t \in \mathcal{T}} \langle[t]_1\rangle$$ where $\langle[t]_1\rangle$ is a free $\F_pG_i$-module for any $t \in \mathcal{T} \cap (K_i^\times \setminus K_{i-1}^\times)$. 
\end{proposition}

\begin{proof}
Following the notation from the previous proof, we have
$$J_1 = \langle[\alpha]_1\rangle \oplus \bigoplus_{y \in \mathcal{Y}_0 \cup \cdots \cup \mathcal{Y}_n} \langle [y]_1\rangle,$$  where each $y \in \mathcal{Y}_i$ has $y \in K_i^\times$ and $\langle[y]_1\rangle$ is a free $\F_pG_i$-module.  As a consequence, it also follows that $y \not\in K_{i-1}^\times$.  Our goal will be to iteratively build sets $\mathcal{T}_j$ for $0 \leq j \leq m-1$ so that for all $j$ we have\begin{enumerate}
\item $J_1 = \langle [\alpha]_1\rangle \oplus \bigoplus_{t \in \mathcal{T}_j} \langle [t]_1\rangle$;
\item $\{\delta_k: \delta_k \neq 0 \text{ and } k \leq j\} \subset \mathcal{T}_j$; and
\item for any $t \in \mathcal{T}_j \cap (K_i^\times \setminus K_{i-1}^\times)$, we have $\langle [t]_1\rangle$ is a free $\F_pG_i$-module.%; and
%\item $|(\{\alpha\} \cup \mathcal{T}_j)\cap (K_i^\times \setminus K_{i-1}^\times)| = e_i(K/F).$
\end{enumerate}  Once we establish this result, clearly $\mathcal{T} = (\mathcal{T}_{m-1} \cup \{\alpha\})\setminus\mathcal{X}_m$ will have the desired properties.

We begin by defining $\mathcal{T}_0 = \mathcal{Y}_0 \cup \cdots \cup \mathcal{Y}_n$.  We have already observed it has the desired properties. %except the final one.  By construction
%\begin{align*}
%|(\{\alpha\} \cup \mathcal{T}_0) \cap (K_i^\times\setminus K_{i-1}^\times)| &= |\mathcal{Y}_i| = \rank_{\F_pG}Y_i \quad \text{ for }i<n\\
%|(\{\alpha\} \cup \mathcal{T}_0) \cap (K_n^\times\setminus K_{n-1}^\times)| &= |\mathcal{Y}_n|+1 = \rank_{\F_pG}Y_n+1.
%\end{align*}

So assume we have constructed $\mathcal{T}_{j-1}$ satisfying the conditions above.  If $\delta_j = 0$, then define $\mathcal{T}_j = \mathcal{T}_{j-1}$, and we are done.  Otherwise, write 
\begin{equation}\label{eq:expressing.delta.j.in.old.basis}
[\delta_j]_1 = g_\alpha [\alpha]_1 + \sum_{t \in \mathcal{T}_{j-1}} g_t [t]_1.
\end{equation}  As in the proof of the last result, there exists some $t_j \in \mathcal{T}_{j-1}$ so that $t_j \in K_{a_j}^\times \setminus K_{a_j-1}^\times$ and $g_{t_j}$ is a unit in $\F_pG$.  In particular, this means we can multiply by some element $\hat g \in \F_pG$ with 
\begin{equation}\label{eq:expressing.old.basis.element.in.new.basis}[t_j]_1 = \hat g\left([\delta_j]_1-g_\alpha[\alpha]_1 - \sum_{t \neq t_j} g_t [t]_1\right).\end{equation}
We claim that $\mathcal{T}_j = \{\delta_j\} \cup (\mathcal{T}_{j-1} \setminus \{t_j\})$ satisfies the desired criteria.

That every $t \in \mathcal{T}_j \cap (K_i^\times\setminus K_{i-1}^\times)$ satisfies $\langle [t]_1\rangle \simeq \F_pG_i$ follows from Proposition \ref{pr:deltafree}.  To show that $\{\delta_k:\delta_k \neq 0 \text{ and }k \leq j\} \subset \mathcal{T}_j$, note that we cannot have $\delta_k = t_j$ for any $k<j$: Proposition \ref{pr:basic.norm.pair.inequality} gives $\delta_k \in K_{a_k}^\times \subset K_{a_j-1}^\times$, and hence $\langle [\delta_k]_1 \rangle \not\simeq \F_pG_{a_j} \simeq \langle [t_j]_1\rangle$.  Finally, we argue that $J_1 = \langle [\alpha]_1\rangle \oplus \bigoplus_{t \in \mathcal{T}_j} \langle[t]_1\rangle$.

First, suppose that $[z]_1 \in J_1$.  By induction we know that there exist $h_\alpha,h_t \in \F_pG$ so that $$[z]_1 = h_\alpha[\alpha]_1 + \sum_{t \in \mathcal{T}_{j-1}} h_t[t]_1.$$  Substituting our expression for $[t_j]_1$ above gives 
\begin{align*}
[z]_1 &= h_\alpha[\alpha]_1 + h_{t_j}[t_j]_1+\sum_{t \in \mathcal{T}_{j-1}\setminus\{t_j\}} h_t[t]_1 \\
&=(h_\alpha-h_{t_j}\hat g g_\alpha)[\alpha]_1 + h_{t_j}\hat g [\delta_j]_1 + \sum_{t \in \mathcal{T}_{j-1}\setminus\{t_j\}} (h_t-h_{t_j}\hat g g_t)[t]_1
\\&\in \langle[\alpha]_1\rangle + \sum_{t \in \mathcal{T}_j} \langle[t]_1\rangle.
\end{align*}

Now we argue that the sum is direct.  Suppose that we have a relation $$[0]_1 = h_\alpha[\alpha]_1 + h_{\delta_j}[\delta_j]_1 + \sum_{t \in \mathcal{T}_j\setminus\{\delta_j\}} h_t [t]_1.$$  If $h_{\delta_j} \in \ann_{\F_pG} \langle[\delta_j]_1\rangle$, then we have a relation amongst a subset of $\{\alpha\} \cup \mathcal{T}_{j-1}$, and hence the relation is trivial by induction. Therefore we may assume $h_{\delta_j} \not\in \ann_{\F_pG}\langle[\delta_j]_1\rangle$.  Substituting equation (\ref{eq:expressing.delta.j.in.old.basis}) into our relation gives
$$[0]_1 = (h_\alpha + h_{\delta_j}g_\alpha)[\alpha]_1 + h_{\delta_j}g_{t_j}[t_j]_1 + \sum_{t \in \mathcal{T}_{j-1}\setminus\{t_j\}} (h_t + h_{\delta_j}g_t)[t]_1.$$  Since $\ann_{\F_pG}\langle[\delta_j]_1\rangle = \ann_{\F_pG}\langle[t_j]_1\rangle$, this gives a nontrivial relation amongst $\{\alpha\} \cup \mathcal{T}_{j-1}$, a contradiction.  
\end{proof}

\begin{corollary}\label{cor:count.on.basis.elements}
For the set $\mathcal{T}$ constructed in Proposition \ref{prop:picking.the.right.basis} we have
    \begin{equation*}
        |\Tc \cap (K_i^\times\setminus K_{i-1}^\times)| = \left\{\begin{array}{ll} e_i(K/F)-\vert\{a_j=i\ :\ 0\le j<m\}\vert&\text{ if }
        i \neq n\\
        e_n(K/F)- \vert\{a_j=n\ :\ 0\le j<m\}\vert-1&\text{ if }i=n.\end{array}\right.
    \end{equation*}
\end{corollary}

\begin{remark*}
Recall that we interpret $K_{i-1}^\times = K_{-\infty}^\times$ when $i = 0$.
\end{remark*}

\begin{proof}
Recall that $\mathcal{T} = (\mathcal{T}_{m-1}\cup \{\alpha\})\setminus\{\alpha,\delta_1,\cdots,\delta_{m-1}\}$.  Hence if we can show that 
$$(\mathcal{T}_{m-1}\cup \{\alpha\})\cap (K_i^\times\setminus K_{i-1}^\times) = \left\{\begin{array}{ll}e_i(K/F) &\text{ if }i \neq a_0\\e_{a_0}(K/F)-1 & \text{ if }i = a_0,\end{array}\right.$$ the desired result will follow.

Note that for the sets $\mathcal{T}_0,\cdots,\mathcal{T}_{m-1}$ constructed in the previous proof, for any given $0 \leq i \leq n$ we have $|(\{\alpha\} \cup \mathcal{T}_j) \cap (K_i^\times\setminus K_{i-1}^\times)|$ does not depend on $j$.  This follows because as we iteratively constructed $\mathcal{T}_j$ from $\mathcal{T}_{j-1}$, either we had $\mathcal{T}_{j-1} = \mathcal{T}_j$ or we removed one element from $K_{a_j}^\times \setminus K_{a_j-1}^\times$ and replaced it with another element from $K_{a_j}^\times \setminus K_{a_j-1}^\times$ (specifically, we replaced $t_j$ with $\delta_j$).  Hence we can focus on studying $$|(\{\alpha\} \cup \mathcal{T}_0) \cap (K_i^\times\setminus K_{i-1}^\times)| = |(\{\alpha\} \cup \mathcal{Y}_0 \cup \cdots \cup \mathcal{Y}_n) \cap (K_i^\times\setminus K_{i-1}^\times)|.$$

We begin by recalling some facts that will be relevant in expressing $e_i(K/F)$.  By Kummer theory 
    \begin{equation*}
        (F^\times \cap pK^\times)/pF^\times = \langle a\rangle,
    \end{equation*}
    where $a$ is any element of $F^\times$ such that
    $K_1=F(\root{p}\of{a})$.  Hence the natural map
    \begin{equation*}
        \theta:F^\times/pF^\times \to J_1
    \end{equation*}
    restricted to $N_{K_i/F}(K_i^\times)+pF^{\times}/pF^\times$ is either injective (in the case that $a\not\in N_{K_i/F}(K_i^\times) \pmod{pF^\times}$) or contains a kernel of order $p$ (otherwise).  For convenience we now adopt the notation $\Theta(S)$ for $S\subset F^\times$ for $\theta(S+pF^{\times})$.

    Now by \cite[Proposition~1]{MSS}, we may select $a_i$, $0\le i<n$, such that $K_{i+1}=K_i(\root{p}\of{a_i})$ and $N_{K_j/K_i}a_j=a_i$ for $0\le i<j\le n-1$.  Since we are not in the case $p=2$, $n=1$, and $-1\not\in N_{K/F}(K^\times)$ we see that $N_{K/F}(\root{p}\of{a_{n-1}})=a_0$.  Since $a_0$ and $a$ generate the same subgroup of $F^\times/pF^\times$, we obtain that $a\in N_{K_i/F}(K_i^\times)\pmod{ pF^\times}$ for each $0\le i\le n$.  Therefore $e_i(K/F)$ equals
    $$\left\{\begin{array}{ll}\displaystyle \dim_{\Fp} \frac{N_{K_i/F}(K_i^\times)+pF^{\times}}{N_{K_{i+1}/F}(K_{i+1}^\times)+pF^{\times}} = \dim_{\Fp} \frac{\Theta(N_{K_i/F}(K_i^\times))}{ \Theta(N_{K_{i+1}/F}(K_{i+1}^\times))}& \text{ if } i<n,\\[15pt]
    %
    %\end{multline*}
    %%and
    %\begin{equation*}
        %e_n(K/F) = 
        1+\dim_{\Fp} \Theta(N_{K/F}(K^\times)) &\text{ if }i=n.\end{array}\right.$$
%    \end{equation*}
    Now by \cite[Cor. 2]{MSS},
    \begin{align*}
        \rank_{\Fp G_i} Y_i &=
        \dim_{\Fp} \frac{\Theta(N_{K_i/F}(K_i^\times))}{
        \Theta(N_{K_{i+1}/F}(K_{i+1}^\times))}, & i<n, i\neq a_0 \\
        1+ \rank_{\Fp G_i} Y_i &= \dim_{\Fp}
        \frac{\Theta(N_{K_{a_0}/F}(K_{a_0}^\times))}{
        \Theta(N_{K_{a_0+1}/F}(K_{a_0+1}^\times))}, & i= a_0 \\
        \rank_{\Fp G} Y_n &= \dim_{\Fp} \Theta(N_{K/F}(K^\times)), &
        i=n.
    \end{align*}
    Hence for $0\le i\le n$,
    \begin{align*}
        \rank_{\Fp G_i} Y_i &=
         e_i(K/F), & i<n, i\neq a_0 \\
        1+\rank_{\Fp G_i} Y_i &=
        e_{a_0}(K/F), & i= a_0 \\
        1+\rank_{\Fp G} Y_n &=
        e_n(K/F), & i=n.
    \end{align*}
    
    But recall that for all $0 \leq i \leq n$ we had $\mathcal{Y}_i \subset K_i^\times$ by construction; moreover we must have $\mathcal{Y}_i \cap K_{i-1}^\times = \emptyset$ since each $y \in \mathcal{Y}_i$ generates a free $\F_pG_i$-module.  Hence
    \begin{equation*}%\label{eq:foundation.cardinalities}
    \begin{split}
    	|(\{\alpha\} \cup \mathcal{T}_0) \cap (K_i^\times\setminus K_{i-1}^\times)| &= \left\{\begin{array}{rll}|\Ycc_i| &= \rank_{\F_pG_i} Y_i &\text{ for }i<n\\ 		|\mathcal{Y}_n|+1 &= \rank_{\F_pG}Y_n+1&\text{ for }i=n.\end{array}\right.
	\end{split}
	\end{equation*}
This gives the desired result.
\end{proof}

\section{Exceptional Modules}\label{sec:exceptional.modules}

The content of this section is to consider the properties of the (multiply-generated) module in $J_m$ determined by a collection of elements which represent a minimal norm pair of length $m$.

\begin{definition*}
  Let $m\in \N$.  An \emph{$m$-exceptional module} is an $R_mG$-module
  \begin{equation*}
        X(\alpha,\{\delta_i\},d,m) = \langle [\alpha]_m,[\delta_1]_m, \cdots,
        [\delta_{m-1}]_m\rangle \subset J_m,
  \end{equation*}
  where $(\alpha,\{\delta_i\})$ represents a minimal norm pair
  $(\mathbf{a},d)$.
\end{definition*}

The most important results of the section are the following two propositions.  For the second result, the reader is encouraged to review the definition of the module $X_{\mathbf{a},d,m}$ from section \ref{sec:main.theorems}.

\begin{proposition}\label{pr:Xsummand}
    Suppose $X$ is an $m$-exceptional module. Then $X$ is a direct
    summand of $J_m$. In particular, suppose $X=X(\alpha,\{\delta_i\},d,m)$, and let $\mathcal{T}$ be the set from Proposition \ref{prop:picking.the.right.basis}. Then
$J_m = X\oplus \sum_{t \in \Tc}
    \langle [t]_m\rangle$.
\end{proposition}

%The following definition appears in \cite{}.  
%For $m\in \N$ and $\mathbf{a}\in \{-\infty,0,1,\dots,n\}^m$, and $d\in
%\Z$, we define the $R_mG$-module
%\begin{align*}
%    X_{\mathbf{a},d,m} := \langle y, x_0, \dots, x_{m-1} :
%    &\ (\sigma-d)y=\sum_{i=0}^{m-1} p^ix_i, \\
%    &\ \sigma^{p^{a_i}}x_i=x_i, \ 0\le i<m\rangle.
%\end{align*}
%It was shown that this module is indecomposable when $\mathbf{a}$ and $d$ satisfy certain hypotheses.

\begin{proposition}\label{pr:isoclass}
    Suppose $X=X(\alpha,\{\delta_i\},d,m)$ is an $m$-exceptional
    module.   Then
    \begin{equation*}
        X \simeq X_{\mathbf{a},d,m}
    \end{equation*}
    under a map induced by $\alpha\mapsto y$, and $\delta_i\mapsto
    x_i$, $0\le i<m$.  Furthermore, $X$ is an indecomposable $R_mG$-module.
\end{proposition}

\begin{proof}[Proof of Proposition~\ref{pr:Xsummand}]
    Write
    \begin{equation*}
        X=X(\alpha,\{\delta_i\},d,m),
    \end{equation*} and let $\mathcal{X}_m = \{\alpha,\delta_1,\cdots,\delta_{m-1}\}$.  To prove that $J_m = X \oplus \sum_{t \in \Tc} \langle [t]_m\rangle$, observe that we have $J_m = X+\sum_{t \in \Tc} \langle [t]_m\rangle$ since $\mathcal{X}_m \cup \Tc$ gives a generating set for $J_1$.  Hence we must only show that $X \cap \sum_{t \in \Tc} \langle [t]_m\rangle = \{[0]_m\}$.  To prove this result, we proceed by induction by showing that for all $1 \leq j \leq m$ we have $[X]_j \cap \sum_{t \in \Tc} \langle [t]_j\rangle = \{[0]_j\}$.  The result holds for $j=1$ by Proposition \ref{prop:picking.the.right.basis}. 
    
    So assume that for some $j \geq 2$ we have $[X]_{j-1} \cap \sum_{t \in \Tc} \langle [t]_{j-1}\rangle = \{[0]_{j-1}\}$, and suppose that $[r]_j \in [X]_j \cap \sum_{t \in \Tc} \langle [t]_{j-1}\rangle$.  Note that by induction we have $[r]_{j-1} = [0]_{j-1}$, so that $[r]_j = p^{j-1}[z]_j$.  Furthermore since $J_j = [X]_j + \sum_{t \in \Tc} \langle [t]_j\rangle$, we can find $[x]_j \in [X]_j$ and $[y]_j \in \sum_{t \in \Tc} \langle [t]_j\rangle$ with $[z]_j = [x]_j + [y]_j$, and so the previous relation is \begin{equation}\label{eq:showing.vanishing.powers.of.X.are.in.X.part.I}[r]_j = p^{j-1}[x]_j + p^{j-1}[y]_j.\end{equation}  We proceed by showing that $p^{j-1}[y]_j = [0]_j$ and that $p^{j-1}[x]_j = [0]_j$, which will give the desired result.
    
    First we argue that $p^{j-1}[y]_j =[0]_j$.  Since $[r]_j \in [X]_j$ we can write $$[r]_j = f_\alpha [\alpha]_j + \sum_{i=1}^{m-1} f_i [\delta_i]_j \quad \quad \text{ for }f_\alpha,f_i \in R_jG.$$      Recalling that $(\sigma-d)\alpha = \sum_{i=0}^{m}p^i \delta_i$ and $d \in U_1$, we can rewrite this relation as
    \begin{equation}\label{eq:showing.vanishing.powers.of.X.are.in.X.part.II}[r]_j = g_\alpha[\alpha]_j + \sum_{i=0}^{m-1} g_i[\delta_i]_j,\end{equation} where now $g_\alpha \in R_j$ and $g_i \in R_jG$. In $J_1$ this relation gives
    $$[0]_1 = g_\alpha[\alpha]_1 + \sum_{i=0}^{m-1} g_i[\delta_i]_1.$$
Recall, however, that $[\alpha]_1$ is $\F_p$-independent from the $\F_pG$-module spanned by $\{[\delta_0]_1,\cdots,[\delta_{m-1}]_1\}$, so that we must have $g_\alpha \in pR_j$.  Write $g_\alpha = ph_\alpha$ for $h_\alpha \in R_j$.
  
  By the independence of the $\F_pG$-modules $\langle[\delta_0]_1\rangle,\langle[\delta_1]_1\rangle,\cdots,\langle[\delta_{m-1}]_1\rangle$, we must have that each $g_i$ reduces to an element of $\ann_{\F_pG}\langle[\delta_i]\rangle$.  Furthermore we know that each $\delta_i \in K_{a_i}^\times$, so $\langle [\delta_i]_1\rangle$ is an $\F_pG_{a_i}$-module; in fact, this module is a free $\F_pG_{a_i}$-module by Proposition \ref{pr:deltafree}.  Hence we have that $g_i$ reduces to $0$ modulo $pR_jG_{a_i}$, and we may find $h_i \in R_jG_{a_i}$ so that $g_i = ph_i$.    
  
  Equations (\ref{eq:showing.vanishing.powers.of.X.are.in.X.part.I}) and (\ref{eq:showing.vanishing.powers.of.X.are.in.X.part.II}) become 
    $$p^{j-1}[x]_{j} + p^{j-1}[y]_j = p\left(h_\alpha[\alpha]_j + \sum_{i=0}^{m-1} h_i[\delta_i]_j\right).$$  Dividing by $p$ and rearranging, we get a relation in $J_{j-1}$:
    $$p^{j-2}[y]_{j-1} = - p^{j-2}[x]_{j-1} + c[\xi_p]_{j-1}+h_\alpha[\alpha]_{j-1} + \sum_{i=0}^{m-1} h_i[\delta_i]_{j-1}.$$  Now the quantity on the left is in $\sum_{t \in \Tc} \langle [t]_{j-1}\rangle$, and --- since $\xi_p = \frac{d-1}{p}\alpha + N_{K{n-1}/F}\delta_0 + \sum_{i=1}^{m} p^i\delta_i$ --- the quantity on the right is in $[X]_{j-1}$.  Since these have trivial intersection in $J_{j-1}$ by induction, we conclude that $p^{j-2}[y]_{j-1} = [0]_{j-1}$, and so $p^{j-1}[y]_{j-1} = [0]_j$.  
    
    Now we argue that $p^{j-1}[x]_j =[0]_j$.  In this case, since $[r]_j \in \sum_{\tau \in \Tc} \langle [\tau]_j\rangle$ we may write \begin{equation}\label{eq:showing.vanishing.powers.of.X.are.in.T}[r]_j = \sum_{t \in \Tc} f_\tau [t]_j\end{equation} where each $f_t \in R_jG$ and all but finitely many $f_t$ are $0$. Considered in $J_1$ this gives a relation $$[0]_1 = \sum_{t \in \Tc} f_t [t]_1.$$  
Since the elements $\{[t]_1:t\in \Tc\}$ generate free, independent summands in $J_1$ by Proposition \ref{prop:picking.the.right.basis}, it follows that for each $t \in \Tc$ we can find $h_t \in R_jG$ so that $f_t = p h_t$.  Combining equations (\ref{eq:showing.vanishing.powers.of.X.are.in.X.part.I}) and (\ref{eq:showing.vanishing.powers.of.X.are.in.T}), extracting $p$th roots and rearranging gives
    $$p^{j-2}[y]_{j-1} + \sum_{t \in \Tc} h_t [t]_{j-1} = c[\xi_p]_{j-1} + p^{j-2}[x]_{j-1}.$$  The elements on the left side are in $\sum_{t \in \Tc} \langle [t]_{j-1}\rangle$ whereas the elements on the right are from $[X]_{j-1}$, and so by induction we know the right side must be zero.  Raising to the $p$th power then gives $p^{j-1}[x]_j = [0]_j$, as desired.
\end{proof}

We now move toward proving Proposition \ref{pr:isoclass}.  Our approach will be induction (on $m$), so it will be useful to know how our exceptional module from $J_m$ decomposes in $J_{m-1}$.

\begin{lemma}\label{le:Xmodm}
    Suppose $m\ge 2$ and $X=X(\alpha,\{\delta_i\},d,m)$
    is an $m$-exceptional module.  Then
    \begin{equation*}
        [X]_{m-1} = X(\alpha,\{\delta_i\},d,m-1) \oplus \langle
        [\delta_{m-1}]_{m-1}\rangle.
    \end{equation*}
\end{lemma}

\begin{proof}
Observe that $\mathcal{X}_{m-1} = \{\alpha,\delta_1,\cdots,\delta_{m-2}\}\setminus\{0\}$ and $\widehat{\mathcal{T}} = \{\delta_{m-1}\} \cup \mathcal{T}$ satisfy the conditions of Proposition \ref{prop:picking.the.right.basis} (when considering $m-1$-exceptional elements).  Applying the previous result we therefore have 
$$X(\alpha,\{\delta_i\},d,m-1) + \sum_{t \in \widehat{\mathcal{T}}}\langle[t]_m\rangle = X(\alpha,\{\delta_i\},d,m-1) \oplus \sum_{t \in \widehat{\mathcal{T}}}\langle[t]_m\rangle.$$
In particular, since $\delta_{m-1} \in \widehat{T}$, we get
\begin{align*}[X]_{m-1} &= \langle[\alpha]_{m-1},[\delta_1]_{m-1},\cdots,[\delta_{m-2}]_{m-1},[\delta_{m-1}]_{m-1}\rangle\\
&= \langle[\alpha]_{m-1},[\delta_1]_{m-1},\cdots,[\delta_{m-2}]_{m-1}\rangle+ \langle [\delta_{m-1}]_{m-1}\rangle\\
&=X(\alpha,\{\delta_i\},d,m-1) \oplus \langle [\delta_{m-1}]_{m-1}\rangle.
\end{align*}
\end{proof}

\begin{lemma}\label{le:deltam-1free}
    Let $m\ge 2$ and suppose that $(\alpha,\{\delta_i\})$ represents
    a minimal norm pair $(\mathbf{a},d)$ of length $m$. Suppose
    $a_{m-1}\neq -\infty$.  Then $\langle [\delta_{m-1}]_{m-1}\rangle$ is a free $R_{m-1}G_{a_{m-1}}$-module.
\end{lemma}

\begin{proof}
    Suppose that $\langle [\delta_{m-1}]_{m-1}\rangle$ is not
    a free $R_{m-1}G_{a_{m-1}}$-module.  Then by
    Lemma~\ref{le:idealrmgiII} we have
    \begin{equation*}
        p^{m-2}P(a_{m-1},0)[\delta_{m-1}]_{m-1} = [0]_{m-1}.
    \end{equation*}
    Taking $p^{m-2}$-th roots, we obtain
    \begin{equation*}
        P(a_{m-1},0)[\delta_{m-1}]_1 = [\xi_{p^{x}}]_1,
    \end{equation*}
    where $\xi_{p^x}$ is a primitive $p^x$-th root of unity for some
    $0\le x\le m-2$.

    By Proposition~\ref{pr:deltafree}, $[\delta_{m-1}]_1$ generates
    a free $\Fp G_{a_{m-1}}$-module, and therefore
    $P(a_{m-1},0)[\delta_{m-1}]_1 \neq [0]_1$.  Then $x\neq 0$,
    and, moreover, no primitive $p^{x+1}$st root of unity lies in
    $K^\times$.  Observe that
    \begin{equation*}
        p^{x-1}P(a_{m-1},0)[\delta_{m-1}]_x = [\xi_{p}]_x.
    \end{equation*}
    Furthermore, this element of $J_{x}$ is nonzero, since otherwise
    $\xi_p$ is a $p^x$-th power in $K^\times$, contradicting the
    fact that no primitive $p^{x+1}$st root of unity lies in
    $K^\times$.  Hence $[0]_x\neq [\xi_p]_x\in \langle [\delta_{m-1}]_x\rangle$.

    From above, $x<m$.  By Corollary~\ref{cor:contacting.exceptional.elements},
    $m$-exceptional elements are $(x+1)$-exceptional and $(\alpha,
    \{\delta_i\}_{i=0}^{x}\cup \{\tilde\delta_{x+1}\})$ represents the minimal norm pair
    $(\mathbf{a}',d)$ of length $x+1$, for some $\tilde\delta_{x+1}\in K^\times$.
    Condition (\ref{eq:defining.properties.for.exceptionality}) tells us that
    \begin{equation*}
        [\xi_p]_x \in \langle [\alpha]_x, [\delta_1]_x,\dots,[\delta_x]_x\rangle
    \end{equation*}
    and so
    \begin{equation*}
        [0]_x\neq [\xi_p]_x \in \langle [\alpha]_x, [\delta_1]_x,\dots,[\delta_x]_x\rangle
        \cap \langle [\delta_{m-1}]_x\rangle.
    \end{equation*}

    On the other hand, since Corollary~\ref{cor:contacting.exceptional.elements} shows that $\alpha$ is $k$-exceptional for all $k<m$, repeated application of Lemma~\ref{le:Xmodm} shows
    \begin{equation*}
        [X(\alpha,\{\delta_i\},d,m)]_x =
        \langle [\alpha]_x,[\delta_1]_x,\cdots,[\delta_{x}]_x \rangle \oplus \langle
        [\delta_{x+1}]_x
        \rangle \oplus \cdots \oplus \langle [\delta_{m-1}]_x \rangle.
    \end{equation*}
    Hence we have a contradiction.
\end{proof}

\begin{proof}[Proof of Proposition~\ref{pr:isoclass}]
    We begin by verifying that $X$ has the appropriate relations amongst its generators.  Let $\chi: M\to X$ be the evaluation homomorphism from
    the free $R_mG$-module $M$ on generators $y,x_0,\dots,x_{m-1}$
    defined by $\chi(y)=\alpha$ and $\chi(x_i)=\delta_i$ for each
    $i=0,\dots, m-1$.  We claim that $I=\ker \chi$ is the ideal
    \begin{align*}
        \tilde I := \left\langle (\sigma-d) y - \sum_{i=0}^{m-1} p^i x_i, (\sigma^{p^{a_i}}-1)x_i \text{ for } i\in\{0, \cdots,
        m-1\}\right\rangle.
    \end{align*}
    We prove this claim by induction on $m$.  In the case $m=1$ the
    structure of
    $X=X(\alpha,\{\delta_i\},d,1)=X(\alpha,\{\delta_i\},1,1)$
    follows from Proposition~\ref{pr:MSS} and \cite[Theorem 2]{MSS}.

    Suppose then that $m\ge 2$ and the proposition holds for $m-1$.
    First we show that $\tilde I \subset I$.  The definition of
    $m$-exceptionality implies that the first generator of $\tilde I$
    lies in $I$, and since $\delta_i\in K_{a_i}^\times$, we have that
    the rest of the generators of $\tilde I$ lie in $I$ as well.
    Hence $\tilde I\subset I$.

    We must therefore show that $I\subset \tilde I$.
    First observe that by definition of $m$-exceptional element,
    \begin{equation*}
      \xi_p = \frac{d-1}{p}N_{K/F}(\alpha)+N_{K_{n-1}/F}(\delta_0)+
      \sum_{i=1}^m p^{i-1}N_{K/F}(\delta_i).
    \end{equation*}
    Hence we define $Z\in M$ by
    \begin{equation*}
      Z = \frac{d-1}{p}P(n,0)y+P(n-1,0)x_0+\sum_{i=1}^{m-1}
      p^{i-1}P(n,0)x_i
    \end{equation*}
    so that $[\chi(Z)]_{m-1}=[\xi_p]_{m-1}$.

    Now suppose we have an
    arbitrary element $W \in I$:
    \begin{equation*}\label{eq:rel1} W := g_\alpha y +
    \sum_{i=0}^{m-1} g_ix_i, \quad g_\alpha, g_i\in R_mG.
    \end{equation*}    Our goal is to show that $W=0 \pmod{\tilde I}$.
    Since
    $$(\sigma-d)y-x_0-\cdots-p^{m-1}x_{m-1}\in\tilde I,$$
    we may assume without loss of generality that $g_\alpha\in R_m$.
    Then we have the following equation in $J_m$:
    \begin{equation}\label{eq:chiofW}
         \chi(W) = g_\alpha [\alpha]_m + \sum_{i=0}^{m-1} g_i[\delta_i]_m=[0]_m.
    \end{equation}

    Now the same relation must hold in $J_1$, so we have
    \begin{equation*}
        g_\alpha [\alpha]_1 + \sum_{i=0}^{m-1} g_i[\delta_i]_1=[0]_1.
    \end{equation*}
    Following the same argument from the proof of Proposition \ref{pr:Xsummand}, this implies that for each $0 \leq i \leq m-1$ we have $g_i \in \langle p,\sigma^{p^{a_i}}-1\rangle$, and that $g_\alpha \in pR_m$.  Since $(\sigma^{p^{a_i}}-1)x_i \in \tilde I$, we may therefore assume that each $g_i \in \langle p\rangle$, and hence can find $h_\alpha,h_0,\cdots,h_{m-1}$ so that $g_\alpha = p h_\alpha$ and $g_i = ph_i$ for all $0 \leq i \leq m-1$.  In particular, the element $\tilde W = h_\alpha y + \sum_{i=0}^{m-1} h_i x_i$ has $W = p \tilde W$.

    Taking $p$th roots of Equation (\ref{eq:chiofW}), we see that
    \begin{equation*}
        h_\alpha[\alpha]_{m-1} + \sum_{i=0}^{m-1}
        h_i[\delta_i]_{m-1} = t[\xi_p]_{m-1}
    \end{equation*}
    for some $t\in \Z$.  Since $[\xi_p]_{m-1} \in [X]_{m-1}$ from condition (\ref{eq:defining.properties.for.exceptionality}), we therefore have a relation
    \begin{equation*}
        h_\alpha[\alpha]_{m-1} + \sum_{i=0}^{m-1}
        h_i[\delta_i]_{m-1} - t[\xi_p]_{m-1} = [0]_{m-1} \in [X]_{m-1}.
    \end{equation*}
    In particular this means that 
    $[\chi(\tilde W-tZ)]_{m-1} = [0]_{m-1}.$
    
    By induction, the kernel of the evaluation map from the free $R_{m-1}G$-module generated by $y,x_0,\cdots,x_{m-2}$ to $\langle[\alpha]_{m-1},[\delta_0]_{m-1},\cdots,[\delta_{m-2}]_{m-1}\rangle$ is $$\left\langle (\sigma-d)y-\sum_{i=0}^{m-2} p^i \delta_i, (\sigma^{p^{a_i}}-1)x_i \text{ for } i \in \{0,\cdots,m-2\}\right\rangle.$$ We also know from Lemmas \ref{le:Xmodm} and \ref{le:deltam-1free} that $$[X]_{m-1} = \langle [\alpha]_{m-1},[\delta_1]_{m-1},\cdots,[\delta_{m-2}]_{m-1}\rangle \oplus \langle [\delta_{m-1}]_{m-1}\rangle,$$ where $\langle[\delta_{m-1}]_{m-1}\rangle$ is a free $R_{m-1}G_{a_{m-1}}$-module.  Together with the fact that $p^{m-1}=0 \in R_{m-1}G$, we therefore have that the kernel of the map from the free module generated by $y,x_0,\cdots,x_{m-1}$ to $[X]_{m-1}$ is precisely
    $$\left\langle (\sigma-d)y-\sum_{i=0}^{m-1} p^i \delta_i, (\sigma^{p^{a_i}}-1)x_i \text{ for } i \in \{0,\cdots,m-1\}\right\rangle.$$
    Hence $\tilde W - tZ$ is in this ideal, and so $p(\tilde W-tZ) = W-tpZ \in \tilde I$.

    It remains only to show that $pZ\in \tilde I$.  Using
    $pP(n-1,0)x_0 = P(n,0)x_0\pmod{\tilde I}$, $(\sigma-1)P(n,0)=0$
    by Lemma~\ref{le:kerbasicII}, and $p^mx_m=0$ since $M$ is an $R_m$-module,
    we calculate
    \begin{align*}
      pZ &= (d-1)P(n,0)y+pP(n-1,0)x_0+\sum_{i=1}^{m}
      p^{i}P(n,0)x_i \\
      &= P(n,0)((d-1)y+\sum_{i=0}^{m}p^ix_i) &\pmod{\tilde I}
      \\
 %     &= P(n,0)((\sigma-1)y-(\sigma-d)y+\sum_{i=0}^m p^ix_i) &\pmod{\tilde I} \\
      &= P(n,0)(\sigma-1)y - P(n,0)((\sigma-d)y-\sum_{i=0}^{m} p^ix_i
      ) &\pmod{\tilde I} \\
      &= 0 &\pmod{\tilde I}.
    \end{align*}
     Therefore we are done.

	We now move on to prove the indecomposability of $X$, for which
it will be enough to verify the hypotheses of Proposition \ref{pr:indecompII}.

    By definition of norm pair, $d\in U_1$.  Hence (\ref{it:exc.mod.indecom.condition...d.in.U1II}) is
    satisfied.

    We turn now to (\ref{it:exc.mod.indecom.condition...power.of.dII}). If $d\in U_m$, then (\ref{it:exc.mod.indecom.condition...power.of.dII}) is satisfied. Otherwise,
    $d\in U_t\setminus U_{t+1}$ for some $1\le t<m$. Suppose first that
    $p>2$ or $d\in U_2$.  Let $i$, $0\le
    i<m$, be arbitrary with $a_i\neq -\infty$. If $i<t$ then since $d\in U_t$ we have
    $d^{p^{a_i}}\in U_t \subset U_{i+1}$.  If $i\ge t$ then by
    Proposition~\ref{pr:nv3}(\ref{it:nv3x}), $a_{i}\ge i-t+1$.  Then
    by Lemma~\ref{le:upowerII}, $d^{p^{a_i}}\in U_{t+a_i} \subset
    U_{i+1}$, as desired.

    We are left with the case $p=2$ and $d\not\in U_2$.  Let $i$,
    $0\le i<m$, be arbitrary with $a_i\neq -\infty$.  If $i=0$ then
    $d^{2^{a_i}}\in U_1$ since $d\in U_1$. Otherwise, by
    Proposition~\ref{pr:basic.norm.pair.inequality}(\ref{it:p.2.d.3.mod.4.means.no.zero.ai}), $a_i> 0$. Now if
    $d=-1$ then by Lemma~\ref{le:upowerII}, $d^{p^{a_i}}=1\in
    U_{i+1}$.  Otherwise, $d\in -U_t\setminus -U_{t+1}$ for some
    $2\le t$. If $i\le t$ then since $a_i>0$ by
    Lemma~\ref{le:upowerII} we have $d^{2^{a_i}}\in U_{t+a_i}\subset
    U_{i+1}$. If $i>t$ then by
    Proposition~\ref{pr:nv3}(\ref{it:nv3x}), $a_i\ge i-t+1$.  Then
    by Lemma~\ref{le:upowerII}, $d^{2^{a_i}} \in U_{t+a_i} \in
    U_{i+1}$, as desired. Hence (\ref{it:exc.mod.indecom.condition...power.of.dII}) is satisfied.

	Condition (\ref{it:exc.mod.indecom.condition...a0.boundsII}) is given by Lemma \ref{lem:p.2.and.n.1.gives.a0.minus.infinity}.

    For (\ref{it:exc.mod.indecom.condition...ai.vs.aj.inequalityII}), we have by
    Proposition~\ref{pr:nv3}(\ref{it:nv3w}) that $a_i+j<a_{i+j}$ for
    all $0\le i<m$, $1\le j< m-i$, and $ a_{i+j}\neq -\infty$ except
    in the case $p=2$, $m\ge 2$, $d\not\in U_2$, and $i=a_i=0$.  In
    this case, however, by
    Proposition~\ref{pr:basic.norm.pair.inequality}(\ref{it:p.2.d.3.mod.4.means.no.zero.ai}), we obtain $a_j\neq
    0$, $1\le j<m$.  Hence (\ref{it:exc.mod.indecom.condition...ai.vs.aj.inequalityII}) is satisfied.

    Finally, (\ref{it:exc.mod.indecom.condition...ai.vs.v.inequalityII}) is
    Proposition~\ref{pr:nv3}(\ref{it:nv3y}).
\end{proof}

\section{Proof of Theorem~{\ref{th:main}}}\label{sec:main.proofs}

%We now have the necessary machinery established to prove the main results.

Let $K/F$ be given, and let $m \in \mathbb{N}$.  By Proposition \ref{pr:norm.pairs.exist}, there exists a minimal norm pair $(\mathbf{a},d)$.  Let $(\alpha,\{\delta_i\}_{i=0}^m)$ represent this minimal norm pair, and let $X=X(\alpha,\{\delta_i\},d,m)$ be the associated $m$-exceptional
module.   Recall that we define $\mathcal{X}_m = \{\alpha,\delta_1,\dots,\delta_{m-1}\}\setminus
    \{0\}$, and by Proposition \ref{prop:picking.the.right.basis} there exists $\Tc\subset K^\times$
such that $$J_1 = \bigoplus_{x \in \Xcc_m}\langle[x]_1\rangle \oplus\bigoplus_{t \in \mathcal{T}} \langle [t]_1\rangle,$$ and so that for each $t \in \mathcal{T} \cap (K_i^\times \setminus K_{i-1}^\times)$ we have $\langle[t]_1\rangle$ is a free $\F_pG_i$-module.  Corollary \ref{cor:count.on.basis.elements} computes $|\mathcal{T} \cap (K_i^\times\setminus K_{i-1}^\times)|$.  Furthermore, Proposition \ref{pr:Xsummand} tells us that $X$ is a direct summand of $J_m$, with its complement generated by classes drawn from $\mathcal{T}$.

We approach the decomposition of Theorem \ref{th:main} by proving the following more precise claim. For each $1\le l\le m$, let
\begin{equation*}
    X_l=\langle [\alpha]_l, [\delta_1]_l, \dots, [\delta_{l-1}]_l \rangle
    \subset J_l,
\end{equation*}
where for $l=1$ we let $X_1=\langle [\alpha]_1\rangle\subset J_1$.
Our claim is that for each $1\le j\le m$,
\begin{enumerate}
    \item\label{it:p1} $X_j\simeq X_{\mathbf{a},d,j}$ and is an indecomposable
    $R_j G$-module;
    \item\label{it:p2} there is a direct complement $W_j$ of $X_j$ in $J_j$
    generated by $\Lcc_j := \Tc\cup (\{\delta_j,
    \dots,\delta_{m-1}\}\setminus \{0\})$
    \item\label{it:p3} for any $0\le i\le n$ and $\gamma \in \Lcc_j \cap
    (K_{i}^\times \setminus K_{i-1}^\times)$, the $R_{j}
    G$-submodule $\langle [\gamma]_j\rangle$ of $J_j$ is a free $R_jG_i$-module
    \item\label{it:p4} the set of $R_{j} G$-submodules $\left\{\langle
    [\gamma]_j\rangle : \gamma \in \Lcc_j\right\}$ is independent
    \item\label{it:p5} for each $0\le i\le n$ we have $$\vert \Lcc_j\cap (K_i^\times\setminus
    K_{i-1}^\times)\vert = \left\{\begin{array}{ll}e_i(K/F)-|\{a_l=i: 0 \leq l < j\}&\text{ if }i<n\\e_n(K/F)- \vert\{a_l=n : 0\le l<j\}\vert -1&\text{ if }i=n.\end{array}\right.$$
\end{enumerate}

Observe that for all $j$, Condition (\ref{it:p5}) comes as an easy consequence of Corollary \ref{cor:count.on.basis.elements}.  We prove the rest of the claims by induction on $j$, starting with
$j=1$. Conditions (\ref{it:p2})--(\ref{it:p4}) above follow from Propositions \ref{pr:deltafree} and \ref{prop:picking.the.right.basis}. ($\mathcal{L}_1$ is what we called $\mathcal{T}_{m-1}$ in  Proposition \ref{prop:picking.the.right.basis}).  Condition (\ref{it:p1}) follows because $X_{\mathbf{a},d,1} \simeq \F_pG/\langle(\sigma-1)^{p^{a_0}+1}\rangle$ by definition, which we know is isomorphic to $\langle[\alpha]_1\rangle$ by Proposition \ref{pr:MSS}.

%By Corollary~\ref{cor:contacting.exceptional.elements}, because $\alpha$
%is $m$-exceptional, $\alpha$ is $1$-exceptional and the minimal norm
%vector $\mathbf{a}$ of length 1 is $(a_0)$. We have $X_1 =
%\langle[\alpha]_1\rangle$.  If $a_0=-\infty$ then by definition of
%$1$-exceptional element, $(\sigma-1)[\alpha]_1=[0]_1$ and therefore
%$X_1$ is a cyclic $\Fp G$-module of length $1$. If $a_0\neq -\infty$
%then by Proposition~\ref{pr:MSS}, $X_1$ is a cyclic $\Fp G$-module
%of dimension $p^{a_0}+1$.  Hence $X_1\simeq X_{\mathbf{a},d,1}$ and $X_1$ is indecomposable (as are all cyclic $\F_pG$-modules).
%Hence we have (\ref{it:p1}).
%
%Now by Lemma~\ref{le:pfoundcont}, $\Lcc_1$ generates a direct
%complement $W_1$ of $X_1$ in $J_1$ with the requirements of
%(\ref{it:p5}). Hence (\ref{it:p2}) is satisfied.  Since $\{\alpha\}
%\cup \Lcc_1$ is a $p$-foundation, the set of $\Fp G$-submodules
%$\left\{\langle [\gamma]_1\rangle:\gamma\in \Lcc_1\right\}$ is %independent,
%giving us (\ref{it:p4}).   Moreover, by
%Proposition~\ref{pr:pfoundfree}, each $\gamma\in \Lcc_1 \cap
%(K_i^\times\setminus K_{i-1}^\times)$ generates a free $\Fp
%G_i$-module, giving us (\ref{it:p3}).
%
Now suppose that for some $j>1$ the claim is true for all $l<j$.  We
show it is true for $j$ as well. Since $X_m=X$ is an $m$-exceptional
module, by definition $(\alpha,\{\delta_i\})$ represents a minimal
norm pair $(\mathbf{a},d)$ of length $m$. Now setting $\hat\delta_k=\delta_k$ for
$k<j$ and $\hat\delta_j=\sum_{l=j}^m p^{l-j}\delta_l$, we see by
Corollary~\ref{cor:contacting.exceptional.elements} that 
$(\alpha,\{\hat\delta_i\})$ represents the minimal norm pair of length $j$ (which we'll also denote as $(\mathbf{a},d)$, though this is a slight abuse of notation).  Therefore
$X_j$ is a $j$-exceptional module. By Propositions~\ref{pr:isoclass}, $X_j$ is an indecomposable $R_j
G$-module isomorphic to $X_{\mathbf{a},d,j}$, giving us (\ref{it:p1}).

Applying Proposition \ref{pr:Xsummand} and Lemma \ref{le:Xmodm} iteratively gives
(\ref{it:p2}).

We show that if $\gamma\in \Lcc_j\cap (K_i^\times \setminus
K_{i-1}^\times)$, then $\langle [\gamma]_j\rangle$ is
a free $R_m G_i$-module. Since $\gamma\in K_i^\times$, we have that $\langle[\gamma]_j\rangle$ is
an $R_m G_i$-module.  If $\langle[\gamma]_j\rangle$ is not free, then by
Lemma~\ref{le:idealrmgiII}, $\ann_{R_jG_i} [\gamma]_j$ contains
$p^{j-1}P(i,0)$.  Therefore
\begin{equation*}
    p^{j-1}P(i,0)[\gamma]_j = [0]_j.
\end{equation*}
%that is, there exists $\beta\in K^\times$ such that
%\begin{equation*}
%    p^{j-1}N_{K_i/F}\gamma = p^j\beta.
%\end{equation*}
Taking $p$th roots gives
\begin{equation*}
    p^{j-2}P(i,0)[\gamma]_{j-1} = c[\xi_p]_{j-1}.
\end{equation*}
On the other hand, $\gamma\in \Lcc_{j-1}$, and from (\ref{it:p3}) by
induction $\langle[\gamma]_{j-1}\rangle$ is a free $R_{j-1}G_i$-module.  By
Lemma~\ref{le:starrmgiII},
\begin{equation*}
    p^{j-2}P(i,0)[\gamma]_{j-1}  \neq [0]_{j-1}.
\end{equation*}
Hence $c[\xi_p]_{j-1}\neq [0]_{j-1}$, and so $c\not\in p\Z$.  If $\tilde c\in \Z$ is chosen such that $c\tilde c=1\pmod{p}$, then we have $p^{j-2}P(i,0)\tilde c[\gamma]_{j-1} = [\xi_p]_{j-1}\neq [0]_{j-1}$.  Now
since $[\gamma]_j \in W_{j}$, we deduce that $[\xi_p]_{j-1}\in [W_{j}]_{j-1}$.

On the other hand, condition (\ref{eq:defining.properties.for.exceptionality}) gives $[\xi_p]_{j-1}\in [X_{j}]_{j-1}$, where by Lemma~\ref{le:Xmodm} we have $[X_j]_{j-1} = X_{j-1} \oplus \langle[\delta_{j-1}]_{j-1}\rangle$. Now since $\Lcc_{j}\subset \Lcc_{j-1}$
and $\delta_{j-1}\in \Lcc_{j-1}\setminus \Lcc_j$, from (\ref{it:p4})
by induction we have that
\begin{equation*}
    W_{j-1} = \langle[\delta_{j-1}]_{j-1}\rangle + [W_j]_{j-1} = \langle[\delta_{j-1}]_{j-1}\rangle\oplus [W_j]_{j-1}.
\end{equation*}
But $W_{j-1}=[W_j]_{j-1}\oplus \langle[\delta_{j-1}]_{j-1}\rangle$ is a direct complement of
$X_{j-1}$ in $J_{j-1}$ by induction, 
%and hence $W_{j}\ (p^{j-1})$ is a direct complement of $X_j\ (p^{j-1})$ in $J_{j-1}$, we have moreover
and so 
\begin{equation}\label{eq:descending.to.jminus1}
    X_{j-1} + \langle[\delta_{j-1}]_{j-1}\rangle \oplus [W_j]_{j-1} = X_{j-1} \oplus
    \langle[\delta_{j-1}]_{j-1}\rangle \oplus [W_j]_{j-1}.
\end{equation}
Having then obtained
\begin{equation*}
    [0]_{j-1}\neq [\xi_p]_{j-1} \in (X_{j-1}\oplus \langle[\delta_{j-1}]_{j-1}\rangle) \cap [W_{j}]_{j-1},
\end{equation*}
we have reached a contradiction.  Hence we have (\ref{it:p3}).

Finally we show that the set $\left\{\langle
[\gamma]_j\rangle: \gamma\in \Lcc_j\right\}$ is independent.  We do
so by induction on the number of such modules.  Since $\langle[\gamma]_j\rangle$ is
a free $R_m G_i$-module for some $i\ge 0$, no $\langle[\gamma]_j\rangle=\{[0]_j\}$, and
therefore the base case is clear.  Suppose then that $\langle[\gamma]_j\rangle \cap
\oplus \langle[\gamma_l]_j\rangle \neq \{0\}$.  By Lemma~\ref{le:starzeroII},
$\langle[\gamma]_j\rangle^\star \cap \oplus \left(\langle[\gamma_l]_j\rangle\right)^\star \neq \{[0]_{j-1}\}$.  Let
$\gamma\in \Lcc_j\cap (K_i^\times\setminus K_{i-1}^\times)$.  By
Lemma~\ref{le:starrmgiII} we conclude that
\begin{equation*}
    [0]_j\neq p^{j-1}P(i,0)[\gamma]_j = \sum p^{j-1}r_l [\gamma_l]_j,\quad r_l\in \Z G.
\end{equation*}
%that is,
%\begin{equation*}
%    0\neq p^{j-1}N_{K_i/F}(\gamma) = \sum p^{j-1}r_l \gamma_l +
%    p^j\beta,\quad \beta\in K^\times.
%\end{equation*}
Taking $p$th roots, we obtain
%\begin{equation*}
%    0\neq p^{j-2}N_{K_i/F}(\gamma) = c\xi_p + \sum p^{j-2}r_l
%    \gamma_l + p^{j-1}\beta, \quad c\in \Z,
%\end{equation*}
%which implies
\begin{equation}\label{eq:pf1eq1}
    -c[\xi_p]_{j-1} = -p^{j-2}P(i,0)[\gamma]_{j-1} + \sum p^{j-2}r_l [\gamma_l]_{j-1}.
\end{equation}
Observe that $\xi_p\in [X_j]_{j-1} = X_{j-1}\oplus \langle[\delta_{j-1}]_{j-1}\rangle$ by Condition (\ref{eq:defining.properties.for.exceptionality}), and that the right side of this expression is clearly an element of $[W_j]_{j-1}$.  By equation (\ref{eq:descending.to.jminus1}), it follows that both sides of \eqref{eq:pf1eq1} equal $[0]_{j-1}$.
In particular,
\begin{equation*}
    [0]_{j-1}\neq p^{j-2}P(i,0)[\gamma]_{j-1} = \sum p^{j-2}r_l [\gamma_l]_{j-1}.
\end{equation*}
Now since $\Lcc_j\subset \Lcc_{j-1}$, we have that  $\left\{\langle[\gamma]_{j-1}\rangle: \gamma\in \Lcc_j\right\}$ is independent in $J_{j-1}$.  Hence we have reached a contradiction and we have (\ref{it:p4}).

\begin{example*}
Return to the example from section \ref{sec:norm.pair.subsection}, where we considered the extension $\mathbb{Q}(\xi_{p^{n+1}})/\mathbb{Q}(\xi_p)$ in the case where $p$ is an odd prime.  Here we will compute the exceptional summand for this extension, as well as the rank of the $Y_0$ summand.  

The computation from our previous example showed that for $\mathbf{a} = (-\infty,\cdots,-\infty)$, and for $d \in \mathbb{Z}$ chosen so that $\sigma(\xi_{p^{n+1}}) = d\xi_{p^{n+1}}$, we have $(\mathbf{a},d)$ is a minimal norm pair.  In light of our proof of Theorem \ref{th:main}, this tells us that the module $J_m$ in this case has its exceptional module isomorphic to $$X_{\mathbf{a},d,m} = \langle y: (\sigma-d)y = 0\rangle / \langle y^{p^m}\rangle.$$  

Now we determine the rank of  $Y_0$  for this extension.  Since $\mathbf{a} = (-\infty,\cdots,-\infty)$, Theorem \ref{th:main} tells us that the rank of $Y_0$ equals $e_0(K/F) = \dim_{\mathbb{F}_p} F^\times F^{\times p}/N_{K_1/F}(K_1^\times)F^{\times p}$.  It is easy to see that $F^{\times}F^{\times p} = F^\times$, and furthermore since any $f^p \in F^{\times p}$ has $f^p = N_{K_1/F}(f)$, we see that the rank of $Y_0$ is simply $\dim_{\mathbb{F}_p} F^\times/N_{K_1/F}(K_1^\times)$.  By \cite[Proposition b]{Pierce} we have that $F^\times/N_{K_1/F}(K_1^\times)$ is isomorphic to the relative Brauer group $\text{\rm{Br}}(K_1/F)$.  But since this group contains an element of order $p$ (indeed, it is exponent $p$), we have $\text{\rm{Br}}(K_1/F) \simeq \bigoplus_{|\mathbb{N}|} \mathbb{Z}/p\mathbb{Z}$ from \cite[Proposition 4]{Fein}.  Hence $Y_0$ has countably infinite rank.
\end{example*}

\end{document}